\documentclass{amsart}

\usepackage{hyperref}
\usepackage{graphicx}
  \graphicspath{{./fig/}}

\usepackage{mathtools} 

\newtheorem*{theo*}{Theorem}
\newtheorem{theo}{Theorem}[section]
\newtheorem{prop}[theo]{Proposition}
\newtheorem{coro}[theo]{Corollary}
\newtheorem{lemm}[theo]{Lemma}
\theoremstyle{remark}
\newtheorem{rema}[theo]{Remark}

\newcommand{\mbf}[1]{\textbf{\textsf{#1}}}

\newcommand{\N}{\mathbb{N}}
\newcommand{\Z}{\mathbb{Z}}
\newcommand{\R}{\mathbb{R}}
\renewcommand{\H}{\mathbb{H}}
\newcommand{\hol}{\operatorname{hol}}
\newcommand{\pr}{\mathrm{pr}}
\renewcommand{\P}{\mathrm{P}}
\newcommand{\tr}{\operatorname{tr}}
\newcommand{\pker}{\operatorname{Pker}}
\newcommand{\SL}{\mathrm{SL}}
\newcommand{\PSL}{\mathrm{PSL}}
\newcommand{\Sp}{\mathrm{Sp}}

\newcommand{\slr}{{\SL(2,\R)}}
\newcommand{\slz}{{\SL(2,\Z)}}
\newcommand{\pslr}{{\PSL(2,\R)}}
\newcommand{\pslz}{{\PSL(2,\Z)}}
\DeclareMathOperator{\sech}{sech}
\DeclareMathOperator{\acosh}{acosh}
\newcommand{\Vol}{\mathrm{A}}
\newcommand{\Cap}{\mathrm{C}}
\newcommand{\KZ}{\mathrm{KZ}}

\renewcommand{\Re}{\operatorname{Re}}
\renewcommand{\Im}{\operatorname{Im}}

\newcommand{\A}{\mathcal{A}}
\newcommand{\B}{\mathcal{B}}
\newcommand{\C}{\mathcal{C}}
\newcommand{\D}{\mathcal{D}}
\newcommand{\M}{\mathcal{M}}

\newcommand{\rquo}[2]{
\mathchoice
{\text{\raisebox{1pt}{$#1$}}{\bigm/}\text{\raisebox{-1pt}{$#2$}}} 
{\text{\raisebox{.5pt}{$#1$}}{\bigm/}\text{\raisebox{-.5pt}{$#2$}}} 
{#1/#2} 
{#1/#2} 
}
\newcommand{\lquo}[2]{
\mathchoice
{\text{\raisebox{-1pt}{$#2$}}{\bigm\backslash}\text{\raisebox{1pt}{$#1$}}} 
{\text{\raisebox{-.5pt}{$#2$}}{\bigm\backslash}\text{\raisebox{.5pt}{$#1$}}} 
{#2\backslash #1} 
{#2\backslash #1} 
}

\begin{document}
\title[Quantitative error term on Veech wind-tree models]{Quantitative error term in the counting problem on Veech wind-tree models}
\author{Angel Pardo}
\address{
  Institut Fourier\\
  Universit\'e Grenoble Alpes\\
  CS 40700\\
  38058 Grenoble cedex 09\\
  France
}
\email{angel.pardo@univ-grenoble-alpes.fr}
\subjclass[2010]{37D50, 37C35; 58J50, 37A40, 37D40, 35P15}
\keywords{Billiards, Translation surfaces, Veech surfaces, Periodic orbits, Counting problem, Error term. Spectrum of Laplace operator, Critical exponent}

\begin{abstract}
We study periodic wind-tree models, billiards in the plane endowed with $\mathbb{Z}^2$-periodically located identical connected symmetric right-angled obstacles. We exhibit effective asymptotic formulas for the number of periodic billiard trajectories (up to isotopy and $\mathbb{Z}^2$-translations) on Veech wind-tree billiards, that is, wind-tree billiards whose underlying compact translation surfaces are Veech surfaces. This is the case, for example, when the side-lengths of the obstacles are rational. We show that the error term depends on spectral properties of the Veech group and give explicit estimates in the case when obstacles are squares of side length $1/2$.
\end{abstract}

\maketitle
\section{Introduction}

The classical wind-tree model corresponds to a billiard in the plane endowed with $\Z^2$-periodic obstacles of rectangular shape aligned along the lattice, as in Figure~\ref{figu:WTM1}.

\begin{figure}[ht]
\centering\includegraphics[height=.17\textheight]{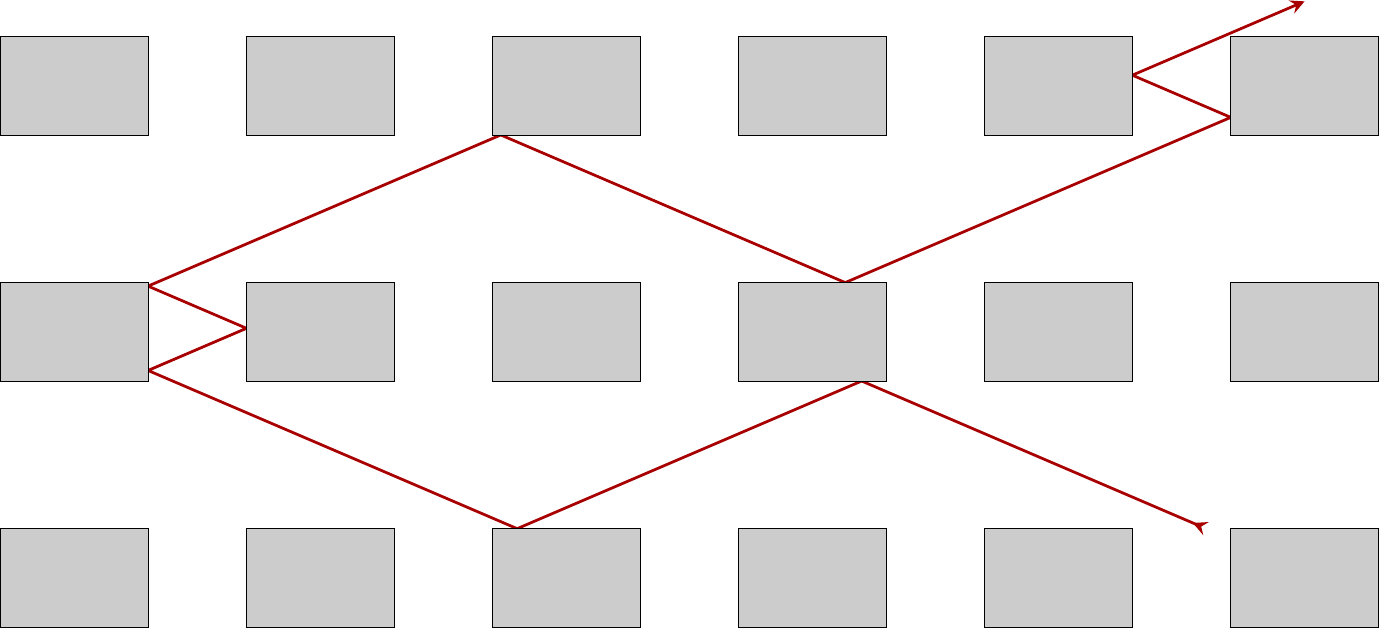}
\caption{Original wind-tree model.}
\label{figu:WTM1}
\end{figure}

The wind-tree model (in a slightly different version) was introduced by P.~Ehrenfest and T.~Ehrenfest~\cite{EE} in 1912. J.~Hardy and J.~Weber~\cite{HW} studied the periodic version. All these studies had physical motivations.

Several advances on the dynamical properties of the billiard flow in the wind-tree model were obtained recently using geometric and dynamical properties on moduli space of (compact) flat surfaces; billiard trajectories can be described by the linear flow on a flat surface.

A.~Avila and P.~Hubert~\cite{AH} showed that for all parameters of the obstacle and for almost all directions, the trajectories are recurrent. There are examples of divergent trajectories constructed by V.~Delecroix~\cite{D}. The non-ergodicity was proved by K.~Fr\c{a}cek and C.~Ulcigrai~\cite{FU}. It was proved by V.~Delecroix, P.~Hubert and S.~Leli\`evre~\cite{DHL} that the diffusion rate is independent either on the concrete values of parameters of the obstacle or on almost any direction and almost any starting point and is equals to $2/3$. A generalization of this last result was shown by V.~Delecroix and A.~Zorich~\cite{DZ} for more complicated obstacles. In the present work we study this last variant, corresponding to a billiard in the plane endowed with $\Z^2$-periodic obstacles of right-angled polygonal shape, aligned along the lattice and horizontally and vertically symmetric. See Figure~\ref{figu:WTMm} for an example.

\begin{figure}[ht]
\centering\includegraphics[height=.19\textheight]{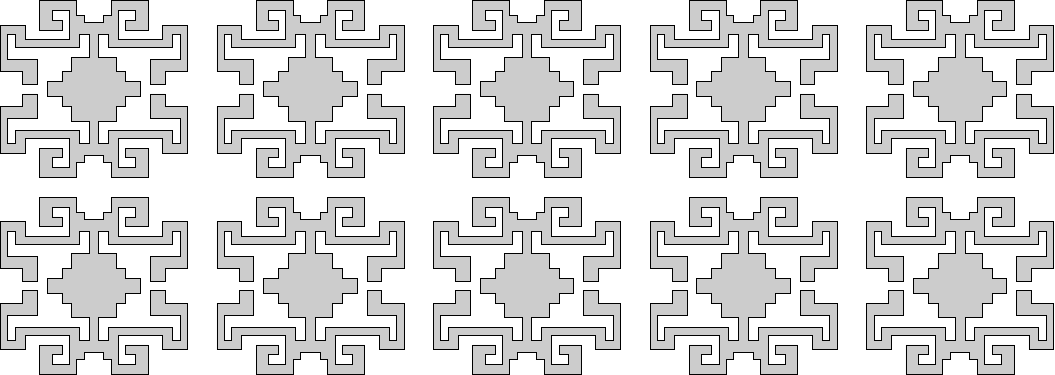}
\caption{Delecroix--Zorich variant.}
\label{figu:WTMm}
\end{figure}

We are concerned with asymptotic formulas for the number of (isotopy classes of) periodic trajectories on the wind-tree model.
This question has been widely studied in the context of (finite area) rational billiards and compact flat surfaces, and it is related to many other questions such as the calculation of the volume of normalized strata~\cite{EMZ} or the sum of Lyapunov exponents of the geodesic Teichm\"uller flow~\cite{EKZ} on strata of flat surfaces (Abelian or quadratic differentials).

H.~Masur~\cite{Ma1,Ma2} proved that for every flat surface (resp. rational billiard) $X$, there exist positive constants $c(X)$ and $C(X)$ such that the number $N(X,L)$ of maximal cylinders of closed geodesics (resp. isotopy classes of periodic trajectories) of length at most $L$ satisfies \[c(X) L^2 \leq N(X,L) \leq C(X) L^2\]
for $L$ large enough.
W.~Veech, in his seminal work~\cite{Ve1}, proved that for Veech surfaces (resp. billiards) there are in fact exact quadratic asymptotics:
\[N(X,L) = c(X) L^2 + o(L^2).\]

Veech surfaces are translation surfaces with a rich group of affine symmetries. They form a dense family on strata, including billiards in regular polygons and square-tiled surfaces.

In this work we study the error term in this kind of asymptotic formulas. In the compact case, the methods used by W.~Veech~\cite{Ve1} give the following result (see \cite[Remark~1.12]{Ve2}).

\begin{theo*}[Veech] 
Let $X$ be a Veech surface. Then, there exists $c(X)>0$ and $\delta(X)\in[1/2,1)$ such that \[N(X,L) = c(X) L^2 + O(L^{2\delta(X)}) + O(L^{4/3})\] as $L\to\infty$.
\end{theo*}

Furthermore, the number $\delta(X)$ has a specific interpretation in terms of spectral properties of the Veech group, the group of derivatives of affine symmetries.

\subsection{Asymptotic formulas for wind-tree models}

In~\cite{Pa}, we proved asymptotic formulas for generic wind-tree models with respect to a natural Lebesgue-type measure on the parameters of the wind-tree billiards, that is, the side lengths of the obstacles (cf.~\cite{AEZ,DZ}) and gave the exact value of the quadratic coefficient, which depends only in the number of corners of the obstacle (see~\cite{Pa} for more details on the counting problem on wind-tree models).
Asymptotic formulas were also given in the case of Veech wind-tree billiards, that is, wind-tree billiards such that the underlying compact translation surface is a Veech surface\footnote{We stress that this notion of ``Veech wind-tree billiard'' is not standard.} (see~\S\ref{sect:wtm} for precise definitions). A concrete set of exemples is when all parameters (the side lengths of the obstacles) are rational. In particular, Veech wind-tree billiards form a dense family.

In the present work, we present an effective version of this result, that is, the analogue of Veech's Theorem
, for Veech wind-tree billiards. 

\begin{theo} \label{theo:quantitative}
Let \,$\Pi$ be a Veech wind-tree billiard. Then, there exists $c(\Pi)>0$ and $\delta(\Pi)\in\left(1/2,1\right)$ such that \[ N(\Pi,L) = c(\Pi) L^2 + O(L^{2\delta(\Pi)}) + O(L^{4/3})\]
as $L\to\infty$.
\end{theo}

This result relies, on one hand, in the adaptation of Veech methods to our context, which allows to keep track one well behaved part of periodic trajectories on wind-tree billiards (\emph{good cylinders}, see \S\ref{sect:cpwtm}). On the other hand, there is a family of badly behaved trajectories (\emph{bad cylinders}, see \S\ref{sect:cpwtm}) which we attack using tools from hyperbolic geometry. Thanks to ideas of F.~Dal'Bo~\cite{Da}, we are able to relate the error term for this family with the Poincar\'e critical exponent of an associated subgroup of the Veech group. We prove then that this critical exponents is strictly less than $1$ using results of R.~Brooks~\cite{Br} (see also~\cite{RT}).

\subsection{Explicit estimates} In the simplest case, when $\Pi$ is the wind-tree billiard with square obstacles of side length $1/2$, the Veech group of $\Pi$ can be easily described and most of the involved objects can be explicitly computed, such as the contribution on the error term of the well behaved part of the periodic trajectories.
Using results of T.~Roblin and S.~Tapie~\cite{RT}, we explicitly estimate the contribution of the badly behaved family of periodic trajectories. More precisely, we prove the following.

\begin{theo} \label{theo:example}
Let \,$\Pi$ be the Veech wind-tree billiard with square obstacles of side length $1/2$, and let $\delta=\delta(\Pi)\in\left(1/2,1\right)$ be as in the conclusion of Theorem~\ref{theo:quantitative}. Then, \[\delta < 0.9885.\]
\end{theo}

\subsection{Strategy of the proof}

W.~Veech~\cite{Ve1} proved that for Veech surfaces there are exact quadratic asymptotics by relating the Dirichlet series of their length spectrum to Eisenstein series associated to the cusps of their (lattice) Veech group. An application of Ikehara's tauberian theorem allows then him to conclude. An effective version of this last tool allows to quantify the error term in terms of spectral properties of the Veech group (see~\cite[Remark~1.12]{Ve2}).

In~\cite{Pa}, we showed that the counting problem on wind-tree models can be reduced to the study of two families of cylinders in the associated translation surface, these are called \emph{good} and \emph{bad} cylinders (see \S\ref{sect:good bad cylinders}, for the precise definition). The notion of good cylinders was first introduced by A.~Avila and P.~Hubert~\cite{AH} in order to give a geometric criterion for recurrence of $\Z^d$-periodic translation surfaces.

Applying Veech's method to the counting problem on Veech wind-tree models, we are able to prove the analogous result in the case of good cylinders, that is, to give the order of the error term in terms of ad-hoc spectral properties of the Veech group of the underlying surface.
This is possible because the collection of good cylinders is $\slr$-equivariant and then, there is a simple description of good cylinders in terms of some particular cusps of the Veech group, which allows to connect the counting problem to the corresponding Eisenstein series as Veech did.

In the case of bad cylinders, this approach does not work anymore since this family is not $\slr$-equivariant and there is no simple description of bad cylinders in terms of (cusps of) the Veech group of the underlying surface. However, bad cylinders can be described in terms of some intricate but well described subgroup $\Gamma_{bad}$ of the Veech group.
Using tools from hyperbolic geometry, thanks to ideas of F.~Dal'Bo~\cite{Da}, we prove that the leading term on the counting of bad cylinders is related to the critical exponent of this subgroup $\Gamma_{bad}$.

Using results of R.~Brooks~\cite{Br}, we prove that this critical exponent is strictly less than $1$. For this, we use the representation of the Veech group given by the restriction of the Kontsevich--Zorich cocycle to a corresponding equivariant subbundle of the real Hodge bundle. The kernel of this representation is a subgroup of $\Gamma_{bad}$. One first application of Brooks results allows us to show that the critical exponents of these two groups coincide. A second application shows that the critical exponent of the kernel is strictly less than that of the Veech group, which equals $1$.

The number $\delta(\Pi)$ in the statement of Theorem~\ref{theo:quantitative}, giving the order of the error term, is completely defined by spectral properties of the involved groups. More precisely, it is the maximum between the critical exponent of the group $\Gamma_{bad}$, associated to bad cylinders, and the second largest pole of the meromorphic continuation of (linear combination of) Eisenstein series, associated to good cylinders. The $4/3$ in the conclusion of Theorem~\ref{theo:quantitative} appears because of technicalities in the effective version of the tauberian theorem for Eisenstein series (\cite[Remark~1.12]{Ve2}).

In the case when $\Pi$ is the wind-tree billiard with square obstacles of side length $1/2$, the Veech group of $\Pi$ is a congruence subgroup of level~$2$. Thanks to a result of M.~Huxley~\cite{Hu}, we know that low level congruence groups satisfies the Selberg's $1/4$ conjecture. To our proposes, this means that the Eisenstein series has no poles in $\left(1/2,1\right)$. 
The critical exponent of $\Gamma_{bad}$ requires much more attention and we are not able to give the exact value. Using results of T.~Roblin and S.~Tapie~\cite{RT}, we estimate the critical exponent of $\Gamma_{bad}$. These estimates are far away from being optimal, but up to our knowledge, this is the only existing tool.

In order to apply this method to estimate the critical exponent of $\Gamma_{bad}$, we have first to give energy estimates on a Dirichlet fundamental domain of the Veech group and to estimate the bottom of the spectrum of the combinatorial Laplace operator associated to the quotient of the Veech group by the above-mentioned kernel.

\subsection{Structure of the paper}
In \S\ref{sect:background} we briefly recall all the background necessary to formulate and prove the results.
In \S\ref{sect:cpvs} we study the counting problem on Veech surfaces associated to collections of cylinders described by a subgroup of the Veech group. We restate Veech's theorem in the case when the subgroup is a lattice and we relate the growth rate to the critical exponent for general subgroups of the Veech group.
In \S\ref{sect:vwtb} we apply this results to the counting problem on Veech wind-tree billiards. Veech's theorem is applied to good cylinders, giving the quadratic asymptotic growth rate with the error term depending in the spectrum of the Veech group. We show that bad cylinders are described by an infinitely generated Fuchsian group of the first kind and prove that its critical exponent is strictly less than one, showing thus the subquadratic asymptotic growth rate of bad cylinders in an effective way.

Finally, in \S\ref{sect:estimates} we study the case of the wind-tree billiard with square obstacles of side length $1/2$. We estimate the critical exponent of the group associated to bad cylinders.
In order to perform this, we give energy estimates in Appendix~\ref{sect:estimates energy} and we estimate the combinatorial specrum of $\pslz$ in Appendix~\ref{sect:estimates combinatorial spectrum}.
Both appendices are self contained and can be read independently of the rest of the paper.

\subsection*{Acknowledgements}
The author is greatly indebted to Pascal Hubert for his guide, constant encouragement, kind explanations and useful discussions. For his invaluable help at every stage of this work. The author is grateful to Vincent Delecroix who, independently to P.~Hubert, take an interest in a quantitative version of the counting problem on wind-tree billiards, their interest being to some extent a first motivation for this work.
The author is thankful to Fran\c{c}oise Dal'Bo for her ideas on how to relate the critical exponent with the asymptotic behavior of the counting function, which are fundamental to this work.
The author is grateful to Samuel Tapie for his kind explanation of his work with T.~Roblin, based on his thesis, on estimates for the critical exponent of a normal subgroup of a lattice group.
The author would like to thank Sebastien Gou\"ezel for the reference to the work of T.~Nagnibeda, where one finds ideas to estimate the bottom of the spectrum of the combinatorial Laplace operator on a Cayley graph.
The author is grateful to Erwan Lanneau for pointing out an error on a computation in a previous version of this work.


\section{Background} \label{sect:background}
\subsection{Rational billiards and translation surfaces}
For an introduction and general references to this subject, we refer the reader to the surveys of Masur--Tabachnikov~\cite{MT}, Zorich~\cite{Zo}, Forni--Matheus~\cite{FM}, Wright~\cite{Wr2}.

\subsubsection{Rational billiards}
Given a polygon whose angles are rational multiples of $\pi$, consider the trajectories of an ideal point mass, that moves at a constant speed without friction in the interior of the polygon and enjoys elastic collisions with the boundary (angles of incidence and reflection are equal). Such an object is called a rational billiard.
There is a classical construction of a translation surface from a rational billiard (see \cite{FK,KZ}).

\subsubsection{Translation surfaces}
Let $g\geq 1$, $\mbf n =\{n_1,\dots,n_k\}$ be a partition of $2g-2$ and $\mathcal{H}(\mbf n)$ denote a stratum of Abelian differentials, that is, holomorphic $1$-forms on Riemann surfaces of genus $g$, with zeros of degrees $n_1,\dots,n_k\in\N$.
There is a one to one correspondence between Abelian differentials and translation surfaces, surfaces which can be obtained by edge-to-edge gluing of polygons in $\R^2$ using translations only. Thus, we refer to elements of $\mathcal{H}(\mbf n)$ as translation surfaces.

A translation surface has a canonical flat metric, the one obtained form $\R^2$, with conical singularities of angle $2\pi(n+1)$ at zeros of degree~$n$ of the Abelian differential.

\begin{rema}\label{rema:marking} A stratum of Abelian differentials $\mathcal{H}(\mbf n)$ has a natural structure of an orbifold. However, using a marking (of horizontal separatrices) we can avoid symmetries which create the orbifold singularities, ensuring a manifold structure on $\mathcal{H}(\mbf n)$. For technical reasons, in this work we consider $\mathcal{H}(\mbf n)$ as a manifold, pointing when the orbifold structure could cause problems.
\end{rema}

\subsubsection{$\slr$-action}
There is a natural action of $\slr$ on strata of translation surfaces,
coming from the linear action of $\slr$ on $\R^2$,
which generalizes the action of $\slr$ on the space $\rquo{\mathrm{GL}(2,\R)}{\slz}$ of flat tori.
Let $g_t = \begin{psmallmatrix} e^t & 0 \\ 0 & e^{-t} \end{psmallmatrix}$; the action of $(g_t)_{t\in\R}$ is called the Teichm\"uller geodesic flow.


\subsubsection{Hodge bundle and the Kontsevich--Zorich cocycle}
The (real) Hodge bundle $H^1$ is the real vector bundle of dimension $2g$ over an affine invariant manifold $\M$ (see~\cite{EMi,EMM} for the precise definition), where the fiber over $X$ is the real cohomology $H^1_X=H^1(X,\R)$. Each fiber $H^1_X$ has a natural lattice $H^1_X(\Z)=H^1(X,\Z)$ which allows identification of nearby fibers and definition of the Gauss--Manin (flat) connection.
The monodromy of the Gauss--Manin connection restricted to $\slr$-orbits provides a cocycle called the Kontsevich--Zorich cocycle, which we denote by $\KZ(A,X)$, for $A\in\slr$ and $X\in\M$.
The Kontsevich--Zorich cocycle is a symplectic cocycle preserving the symplectic intersection form $\langle f_1,f_2\rangle=\int_S f_1\wedge f_2$ on $H^1(X,\R)$.

\subsubsection{Lyapunov exponents}
Given any affine invariant manifold $\M$, we know from Oseledets theorem that there are real numbers $\lambda_1(\M) \geq \dots \geq \lambda_{2g}(\M)$, the Lyapunov exponents of the Kontsevich--Zorich cocycle over the Teichm\"uller flow on $\M$
and a measurable $g_t$-equivariant filtration of the Hodge bundle $H^1(X,\R) = V_1(X)\supset\dots\supset V_{2g}(X)=\{0\}$ at $\nu_\M$-almost every $X\in\M$ such that \[\lim_{t\to\infty} \frac{1}{t} \log \| \KZ(g_t,X) f\|_{g_t\omega}=\lambda_i\] for every $f\in V_i\setminus V_{i+1}$.

The fact that the Kontsevich--Zorich cocycle is symplectic implies that the Lyapunov spectrum is symmetric, $\lambda_j = -\lambda_{2g-j}$, $j=0,\dots,g$.

\subsubsection{Equivariant subbundles of the Hodge bundle} \label{sect:subbundles}
Let $\M$ be an affine invariant submanifold and $F$ a subbundle of the Hodge bundle over $\M$. We say that $F$ is equivariant if it is invariant under the Kontsevich--Zorich cocycle.
Since $\M$ is $\slr$-invariant, by the definition of the Kontsevich--Zorich cocycle, a flat (locally constant) subbundle is always equivariant.

We say that $F$ admit an almost invariant splitting, if there exists $n \geq1$ and for $\nu_\M$-almost every $X\in\M$ there exist proper subspaces $W_1(X),\dots,W_n(X)\subset F_X$ such that $W_i(X)\cap W_j(X) = \{0\}$ for $1\leq i < j \leq n$, such that, for every $i\in\{1,\dots,n\}$ and almost every $A\in\slr$,
$\KZ(A,X)W_i(X) = W_j(AX)$ for some $j\in\{1,\dots,n\}$, and such that the map $X\mapsto\{W_1(X),\dots, W_n(X)\}$ is $\nu_\M$-measurable.
We say that $F$ is strongly irreducible if is does not admit an almost invariant splitting.

\begin{rema}\label{rema:vector bundle} Without avoiding symmetries which causes orbifold points on $\mathcal{H}(\mbf n)$ (see Remark~\ref{rema:marking}), the Hodge bundle would not be an actual vector bundle (we would have to consider the cohomology group up to symmetries) and the Kontsevich--Zorich cocycle would not be an actual linear cocycle. In this work we consider some invariant splittings of the Hodge bundle which would not be invariant by the whole action of $\slr$ if we do not consider the marking.
\end{rema}

Previous discussion about Lyapunov exponents applies in this context as well and we have that, as before, for almost every $X\in\M$, there is a measurable $g_t$-equivariant filtration $F_{X} = U_1(X)\supset\dots\supset U_r(X)=\{0\}$, where $r=\operatorname{rank}F=\dim F_X$ and, for every $f\in U_i\setminus U_{i+1}$, \[\lim_{t\to\infty} \frac{1}{t} \log \| \KZ(g_t,X) f\|_{g_tr_\theta\omega}=\lambda_i(\M,F).\]

We denote by $F_X(\Z)= F_X\cap H^1_X(\Z)$ the set of integer cocycles in $F_X$.
We say that $F$ is defined over $\Z$ if it is generated by integer cocycles, that is, if $F_X=\left<F_X(\Z)\right>_\R$.
When $F$ is defined over $\Z$, $F_X(\Z)$ is a lattice in $F_X$.

\subsubsection{Veech group and Veech surfaces}
We denote the stabilizer of a translation surface $X$ under the action of $\slr$ by $\SL(X)$.
The group $\SL(X)$ is also the group of derivatives of affine orientation-preserving diffeomorphisms of $X$.

Recall that $\slr$ does not act faithfully on the upper half-plane $\H$; it is the projective group $\pslr$ that does so. If $G$ is a subgroup of $\slr$, we denote by $\P G$ its image in $\pslr$. In a slight abuse of notation we sometimes shall omit $\P$ whenever it is clear from the context that we see $G$ as a subgroup of $\slr$ or $\pslr$.
We define the \emph{Veech Group} of $X$ to be $\PSL(X)$, that is, the image of $\SL(X)$ in $\pslr$.

A translation surface $X$ is called \emph{Veech surface} if its Veech group $\PSL(X)$ is a lattice, that is, if $\rquo{\H}{\PSL(X)}$ has finite volume. Veech surfaces correspond to closed $\slr$-orbits. Such a closed orbits is called a Teichm\"uller curve.
In this work we are devoted to Veech surfaces.
For an introduction and general references to Veech surfaces, we refer the reader to the survey of Hubert--Shcmidt~\cite{HS}.

\begin{rema}\label{rema:elliptic elements} Since we are considering markings on translation surfaces in order to avoid orbifold points on strata (see Remark~\ref{rema:marking}), elliptic elements (that is, finite order elements) are never in $\PSL(X)$.
\end{rema}

\subsubsection{Veech group representation} \label{sect:representation}
When $A\in\SL(X)$, the Kontsevich--Zorich cocycle defines a symplectic map $\KZ(A,X):H^1_X\to H^1_X$ which preserves $H^1_X(\Z)$. This defines thus a representation $\rho_{H^1}$ of $\SL(X)$ on the symplectic group $\Sp(H^1_X,\Z)$,
\[
\begin{matrix}
\rho_{H^1}: & \SL(X) & \to & \Sp(H^1_X,\Z), \\
& A & \mapsto & \KZ(A,X).
\end{matrix}
\]
If $F$ is an equivariant subbundle, then the restriction of the Kontsevich--Zorich cocycle to $F$ gives another representation which, in general, is not faithful and we denote it by $\rho_{F}: \SL(X) \to \SL(F_X)$. Note that, in general, this representation is neither symplectic nor defined over $\Z$. However, if the subbundle is symplectic or defined over $\Z$, so is the representation.

Since, by our convention, finite order elements are not allowed in $\SL(X)$, in particular $-id\notin \SL(X)$ and hence every representation $\rho_F$ descends to a representation of $\PSL(X)$ on $\PSL(F_X)$.

\subsection{Counting problem}
We are interested in the counting of closed geodesics of bounded length on translation surfaces.

\subsubsection{Cylinders}
Together with every closed regular geodesic in a translation surface $X$ we have a bunch of parallel closed regular geodesics.
A cylinder on a translation surface is a maximal open annulus filled by isotopic simple closed regular geodesics.
A cylinder $C$ is isometric to the product of an open interval and a circle, and its core curve $\gamma_C$ is the geodesic projecting to the middle of the interval.

\subsubsection{Holonomy}
Integrating the corresponding Abelian differential along the core curve of a cylinder or, more generally, any homology class $\gamma\in H_1(X,\Z)$, we get a complex number. Considered as a planar vector, it represents the affine holonomy along $\gamma$ and we denote this holonomy vector by $\hol(\gamma)$. In particular, the euclidean length of a cylinder corresponds to the modulus of its holonomy vector.

A relevant equivariant subbundle is given by $\ker \hol$ which in turn is the symplectic complement of the so called tautological (sub)bundle.

\subsubsection{Counting problem}
Consider the collection of all cylinders on a translation surface $X$ and consider its image $V(X)\subset\R^2$ under the holonomy map, $V(X)=\left\{ \hol \gamma_C : C \text{ is a cylinder in } X \right\}$. This is a discrete set of $\R^2$. We are concerned with the asymptotic behavior of the number $N(X,L)=\# V(X)\cap B(L)$ of cylinders in $X$ of length at most $L$, when $L\to\infty$.

More generally, we can consider any collection of cylinders $\C\subset\A$, and study the asymptotic behavior of the number of cylinders in $\C$ of length at most $L$, $N_\C(X,L)=\# V_\C(X)\cap B(L)$, as $L\to\infty$, where $V_\C(X)=\left\{\hol\gamma_C:C\in\C\right\}$.

\subsection{Wind-tree model} \label{sect:wtm}
The wind-tree model corresponds to a billiard $\Pi$ in the plane endowed with $\Z^2$-periodic horizontally and vertically symmetric right-angled obstacles, where the sides of the obstacles are aligned along the lattice as in Figure~\ref{figu:WTM1} and Figure~\ref{figu:WTMm}.

Recall that in the classical case of a billiard in a rectangle we can glue a flat torus out of four copies of the billiard table and unfold billiard trajectories to flat geodesics of the same length on the resulting flat torus.
In the case of the wind-tree model we also start from gluing a translation surface out of four copies of the infinite billiard table $\Pi$. The resulting surface $X_\infty=X_\infty(\Pi)$
is $\Z^2$-periodic with respect to translations by vectors of the original lattice. Passing to the $\Z^2$-quotient we get a compact translation surface $X=X(\Pi)$. For the case of the original wind-tree billiard, with rectangular obstacles, the resulting translation surface is represented at Figure~\ref{figu:compact-surface}
(see \cite[\S~3]{DHL} for more details).

\begin{figure}[ht]
\centering\includegraphics[height=.2\textheight]{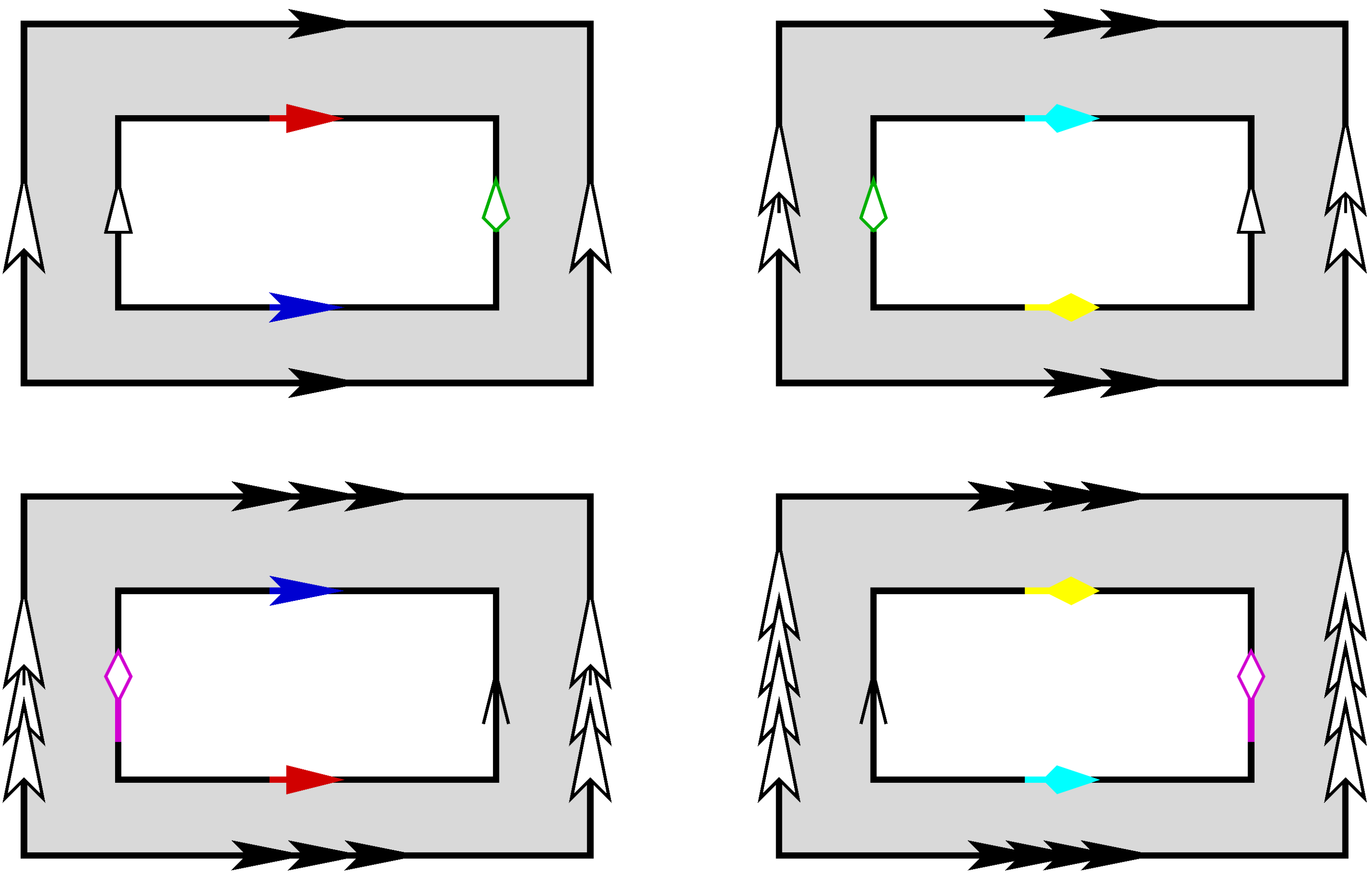}
\caption{The translation surface $X$ obtained as quotient over $\Z^2$ of an unfolded wind-tree billiard table (\cite[Figure~5]{DZ}).}
\label{figu:compact-surface}
\end{figure}

Similarly, when the obstacle has $4m$ corners with the angle $\pi/2$ (and therefore, $4m-4$ with angle $3\pi/2$), the same construction gives a translation surface 
consisting in four flat tori with holes ---four copies of a $\Z^2$-fundamental domain of $\Pi$, the holes corresponding to the obstacles--- with corresponding identifications, as in the classical setting ($m=1$, see Figure~\ref{figu:compact-surface}).

\subsubsection{Description of the $\Z^2$-covering and relevant subbundles} \label{sect:wind-tree subbundles}
There are two cohomology classes $h,v\in H^1(X,\Z)$ defining the $\Z^2$-covering $X_\infty$ of $X$.
Let $\M$ be the $\slr$-orbit closure of $X$. Then, thanks to the symmetries of $X$, there are two equivariant subbundles $F^{(h)}$ and $F^{(v)}$ of $H^1$ defined over $\M$, such that $h\in F^{(h)}$ and $v\in F^{(v)}$ (see~\cite{Pa} for more details). Furthermore, we have the following (see~\cite[Corollary~5]{Pa}).

\begin{theo} \label{theo:positive Lyapunov exponent}
Let $\Pi$ be a wind-tree billiard, $X=X(\Pi)$. Then, the subbundles $F^{(h)}$ and $F^{(v)}$ defined over the $\slr$-orbit closure of $X$ are $2$-dimensional flat subbundles defined over $\Z$ and have non-zero Lyapunov exponents.
\end{theo}

As consequence, these subbundles are strongly irreducible and symplectic. Indeed, by~\cite[Theorem~1.4]{AEM} and \cite[Theorem~A.9]{EMi}, any measurable equivariant subbundle with at least one non-zero Lyapunov exponent is symplectic and, in particular, even dimensional. Thus, a two-dimensional subbundle is automatically strongly irreducible provided it has non-zero Lyapunov exponents. Furthermore, these subbundles are subbundles of $\ker\hol$.

\subsubsection{The $(1/2,1/2)$ wind-tree model} \label{sect:example} We give a little more details in the case of the wind-tree billiard with square obstacles of side length $1/2$, $\Pi = \Pi(1/2, 1/2)$.

\begin{figure}[ht]
\centering\includegraphics[height=.3\textheight]{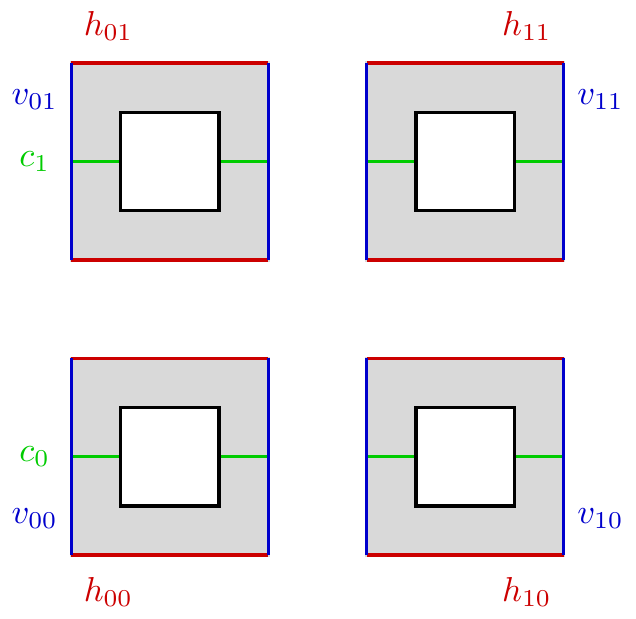}
\caption{The surface $X=X(\Pi(1/2,1/2))$ and the cycles $h_{ij}$, $v_{ij}$ and $c_j$, $i,j \in \{0,1\}$ (cf.~\cite[Figure~4]{DHL}).}
\label{figu:homology basis}
\end{figure}

The surface $X=X(\Pi)$ is a covering of a genus~$2$ surface $L$ which is a so called L-shaped surface that belongs to the stratum $\mathcal{H}(2)$ (see for example \cite{DHL}). In particular, $\SL(X)$ is a finite index subgroup of $\SL(L)$.
In this particular case, $L$ is a square-tiled surface, tiled by $3$ squares, as in Figure~\ref{figu:L-shaped}.

It is elementary to see that the stabilizer of $L$ is generated by $r=\begin{psmallmatrix} 0 & 1 \\ -1 & 0 \end{psmallmatrix}$ and $u^2=\begin{psmallmatrix} 1 & 2 \\ 0 & 1 \end{psmallmatrix}$ (see for example \cite[\S9.5]{Zo}).
However, with our convention on markings, elliptic elements are forbidden (see Remark~\ref{rema:elliptic elements}) and thus, $\SL(L)=\langle u^2,{}^{t\!}u^2 \rangle$, where ${}^{t\!}g$ is the transpose of $g$.
Moreover, it is not difficult to verify that $\SL(X)=\SL(L)$. In particular, $\PSL(X)$ is a level two congruence group.

\begin{figure}[ht]
\centering
\raisebox{-.5\height}{\includegraphics[width=.9\textwidth,height=.25\textheight,keepaspectratio]{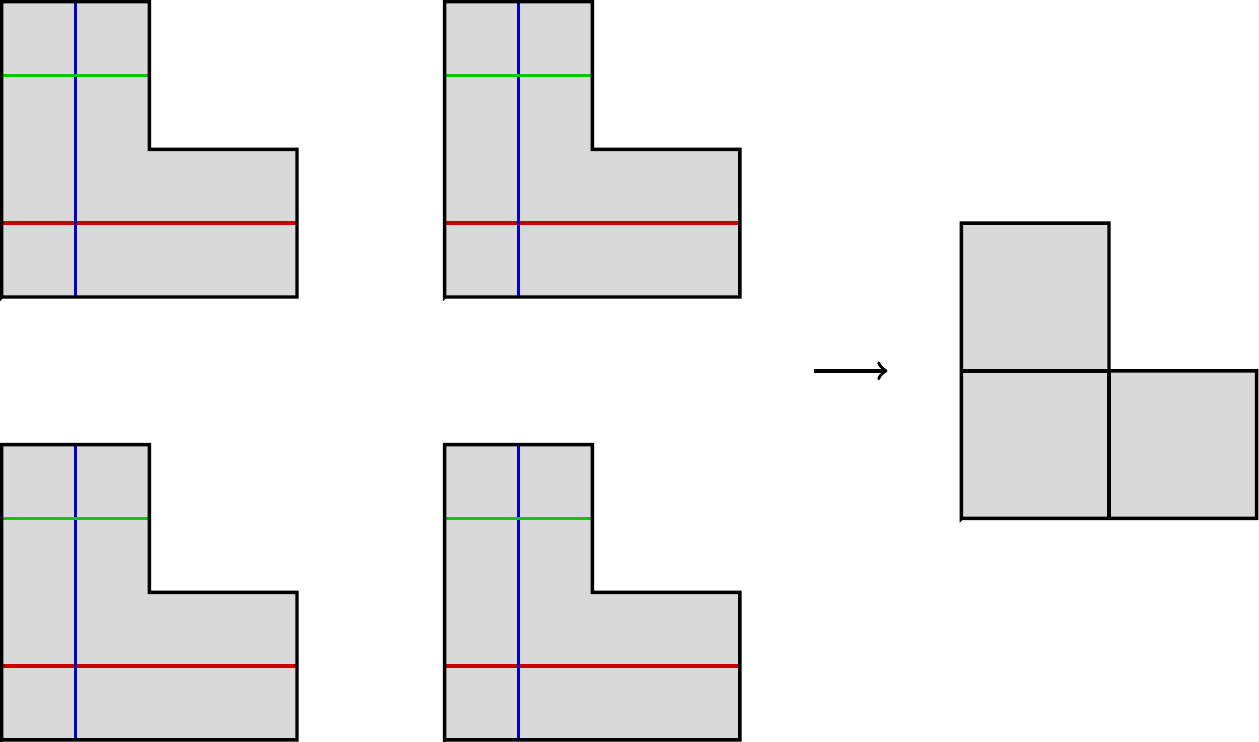}}
\caption{The surface $X=X(\Pi(1/2,1/2))$ seen as a cover of the square-tiled L-shaped surface $L$.}
\label{figu:L-shaped}
\end{figure}

For $i,j \in \{0,1\}$, let $h_{ij}$, $v_{ij}$ and $c_j$ be as in Figure~\ref{figu:homology basis}.
Let $E^{+-}$ be the subspace of $H^1(X,\R)$ with symplectic integer basis $\{h^{+-},v^{+-}\}$, where $h^{+-}$ is the Poincar\'e dual of the cycle $h_{00}+h_{01}-h_{10}-h_{11}$ and $v^{+-}$, of $v_{00}+v_{01}-v_{10}-v_{11}$.
Similarly, define $E^{-+}$, with basis $\{h^{-+},v^{-+}\}$, where $h^{-+}=(h_{00}-h_{01}+h_{10}-h_{11})^*$ and $v^{-+}=(v_{00}-v_{01}+v_{10}-v_{11})^*$.

In our notation, we have that $F^{(h)}_X = E^{+-}$, $h = h^{+-}$, $F^{(v)}_X=E^{-+}$ and $v = v^{-+}$.

The action of $u^2\in\SL(X)$ on the $h_{ij},v_{ij}$, $i,j \in \{0,1\}$ is shown in Figure~\ref{figu:u2-action} and is described by
\begin{align*}
\rho_{H^1}(u^2): \quad
h_{ij}^* & \mapsto h_{ij}^* \\
v_{ij}^* & \mapsto v_{ij}^* + h_{ij}^* + c_j^*.
\end{align*}
Denoting $c^{+-} \coloneqq 2c_0 - 2c_1$ we have that $c^{+-} = 2h^{+-}$. Letting $c^{-+} \coloneqq 0$ we obtain that, for $\sigma\in\{+-,-+\}$,
\begin{align*}
\rho_{E^\sigma}(u^2): \quad
h^\sigma & \mapsto h^\sigma \\
v^\sigma & \mapsto v^\sigma + h^\sigma + c^\sigma.
\end{align*}

\begin{figure}[ht]
\centering\includegraphics[width=\textwidth,height=.25\textheight,keepaspectratio]{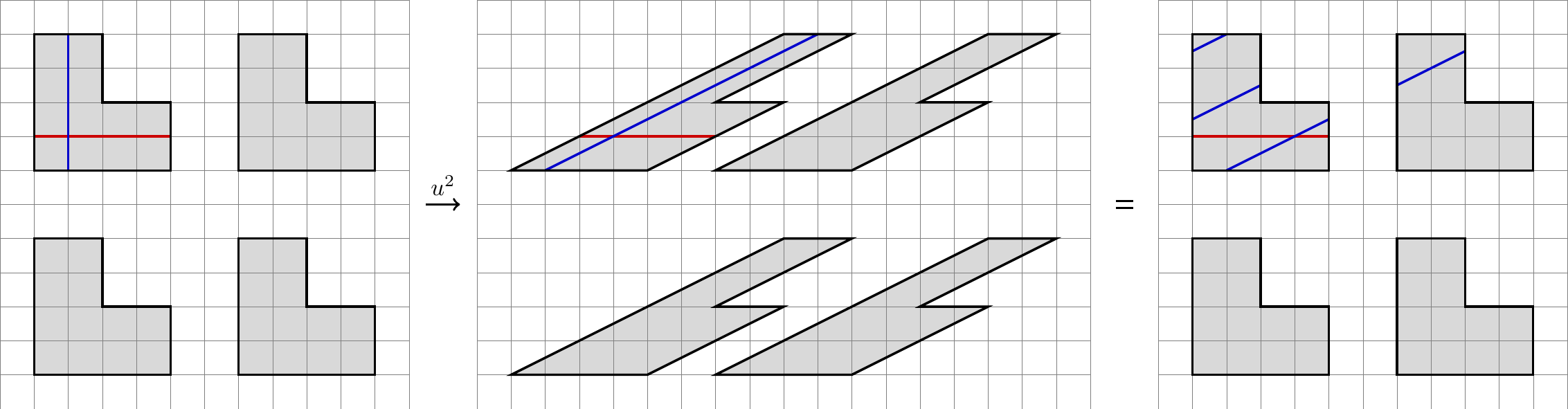}
\caption{The action of $u^2$ on $h_{ij},v_{ij}$, $i,j\in\{0,1\}$.}
\label{figu:u2-action}
\end{figure}

Thus, with the choice of basis as above, we get $\rho_{E^{+-}}(u^2) = u^3$ and $\rho_{E^{-+}}(u^2) = u$. Similarly, we can see that $\rho_{E^{+-}}({}^{t\!}u^2) = {}^{t\!}u$ and $\rho_{E^{-+}}({}^{t\!}u^2) = {}^{t\!}u^3$. In particular, the two representations are isomorphic $\rho_{F^{(h)}}(\SL(X))\cong \rho_{F^{(v)}}(\SL(X))$.

\subsection{Counting problem on wind-tree models} \label{sect:cpwtm}
In this work, we are concerned with counting periodic trajectories in the wind-tree billiard. Obviously, any periodic trajectory can be translated by an element in $\Z^2$ to obtain a new (non-isotopic) periodic trajectory. Then, we shall count (isotopy classes of) periodic trajectories of bounded length in the wind-tree billiard, up to $\Z^2$-translations.

There is a one to one correspondence between billiard trajectories in $\Pi$ and geodesics in $X_\infty$. But $X_\infty$ is the $\Z^2$-covering of $X$ given by $h,v\in H^1(X,\Z)$, which means that closed curves $\gamma$ in $X$ lift to closed curves in $X_\infty$ if and only if $h(\gamma_C)=v(\gamma_C)= 0$.
In fact, by definition of the covering, the monodromy of a closed curve $\gamma$ in $X$ is the translation by $(h(\gamma),v(\gamma))\in\Z^2$. The cylinders in the cover $X_\infty$ are exactly the lift of those cylinders $C$ in $X$ whose core curve $\gamma_C$ has trivial monodromy. In particular, cylinders in $X_\infty$ are always isometric to their projection on $X$. When a cylinder $C$ does not satisfy this condition, it lifts to $X_\infty$ as a strip, isometric to the product of an open interval and a straight line.

\subsubsection{Good and bad cylinders} \label{sect:good bad cylinders}
Let $f=h$ or $v$, and $F=F^{(f)}$.
Note that cylinders $C$ in $X$ such that $f(\gamma_C) = 0$, split naturally into two families:
(a) the family of cylinders such that $\hat f(\gamma_C) = 0$ for all $\hat f\in F_X$, that is, $\gamma_C\in\mathrm{Ann}(F_X)$, which we call \emph{$F$-good cylinders}, and
(b) the family of cylinders that are not $F$-good, but $f(\gamma_C) = 0$. These later are called \emph{$(F,f)$-bad cylinders}.
The notion of $F$-good cylinders was first introduced by Avila--Hubert~\cite{AH} in order to give a geometric criterion for recurrence of $\Z^d$-periodic flat surfaces. Good cylinders are favorable to our purposes. In fact, since the Kontsevich--Zorich cocycle preserves the intersection form and $F$ is equivariant, they define an $\slr$-equivariant family of cylinders, which is much more tractable than arbitrary collections of cylinders.

For a wind-tree billiard $\Pi$, we denote by $N(\Pi,L)$, the number of (isotopy classes of) periodic trajectories (up to $\Z^2$-translations) of length at most $L$, by $N_{good}(X,L)$ the number of $F^{(h)}\oplus F^{(v)}$-good cylinders in $X=X(\Pi)$ of length at most $L$ and $N_{f-bad}(X,L)$, of $(F,f)$-bad cylinders in $X$ of length at most $L$, for $f=h$ or $v$ and $F=F^{(f)}$.

Note that \[N_{good}(X,L) \leq N(\Pi,L) \leq N_{good}(X,L) + N_{h-bad}(X,L) + N_{v-bad}(X,L).\]
Therefore, it is enough to understand the asymptotic behavior of $N_{good}(X,L)$, $N_{h-bad}(X,L)$ and $N_{v-bad}(X,L)$ separately.

The author~\cite{Pa} used this to reduce the counting problem on wind-tree models to the counting of good cylinders. In fact, we have the following (see~\cite[Theorem~1.3]{Pa}).

\begin{theo*} 
Let $\Pi$ be a wind-tree billiard, $X=X(\Pi)$ the associated compact flat surface, let $f=h$ or $v$ and $F=F^{(f)}$ be one of the associated subbundles $F^{(h)}$ or $F^{(v)}$. Then, the number $N_{f-bad}(X,L)$, of $(F,f)$-bad cylinders in $X$ of length at most $L$, has \emph{subquadratic} asymptotic growth rate, that is, $N_{f-bad}(X,L)=o(L^2)$.
\end{theo*}

Thus, the counting problem on wind-tree models may be reduced to count $F^{(h)}\oplus F^{(v)}$-good cylinders, which has quadratic asymptotic growth rate thanks to a result of Eskin--Masur~\cite{EMa}.
However, in this work, we are interested in an effective version and therefore, bad cylinders have to be taken into account.

\begin{rema} \label{rema:cylinder projections}
An useful characterization of bad cylinders in our case is the following. A cylinder $C$ is $(F,f)$-bad if and only if $\pr_{F_X} \gamma_C = \pm f$. In fact, since $F$ is symplectic and two dimensional, $C$ is an $(F,f)$-bad cylinder if and only if $\pr_{F_X} \gamma_C \neq 0$ is colinear to $f$ (see~\cite[Remark~3.1]{Pa}). Moreover, the action of $\slr$ on homology (that is, the Kontsevich--Zorich cocycle) is by integer matrices, then, this is equivalent to say that $\pr_{F_X} \gamma_C = \pm f$.
\end{rema}

\subsubsection{Veech wind-tree billiards} \label{sect:vwtm}
Let $\Pi$ be a wind-tree billiard.
We define the \emph{Veech group of $\Pi$} to be $\PSL(\Pi)=\PSL(X(\Pi))$ and we say that $\Pi$ is a \emph{Veech wind-tree billiard} if $\PSL(\Pi)$ is a lattice. We stress that these definitions are not standard as it does not correspond to the (projection to $\pslr$ of the) derivatives of affine orientation-preserving diffeomorphisms of the unfolded billiard $X_\infty(\Pi)$, but to those of $X(\Pi)$, the $\Z^2$-quotient of the unfolded billiard.

In the classical case, of rectangular obstacles, we denote $\Pi(a,b)$ the wind-tree billiard with rectangular obstacles of side lengths $a,b\in\left]0,1\right[$. Thank to results of Calta~\cite{Ca} and McMullen~\cite{McM1,McM2}, it is possible to classify completely Veech wind-tree models in the classical case (see~\cite[Theorem~3]{DHL}).

\begin{theo*}[Calta, McMullen]
The wind-tree model $\Pi(a, b)$ is a Veech wind-tree billiard if and only if either
$a, b \in \mathbb{Q}$ or there exist $x,y\in \mathbb{Q}$ and a square-free integer $D>1$ such that $1/(1-a)=x+y\sqrt{D}$ and $1/(1-b)=(1-x)+y\sqrt{D}$.
\end{theo*}

In this work we are concerned only with Veech wind-tree billiards. Most of the tools we use to deal with bad cylinders comes from geometric considerations of the action (on the upper half-plane $\H$) of the lattice Veech group $\PSL(\Pi)$ and, more precisely, of some particular subgroups of $\PSL(\Pi)$. These groups are Fuchsian groups. In the following, we present a brief recall of the objects we need and some of their properties.

\subsection{Fuchsian groups}
A Fuchsian group is a discrete subgroup of $\pslr$. 
A Fuchsian group $\Gamma$ acts properly discontinuously on $\H$. In particular, the orbit $\Gamma z$ of any point $z\in\H$ under the action of $\Gamma$ has no accumulation points in $\H$. There may, however, be limit points on the real axis. Let $\Lambda(\Gamma)$ be the limit set of $\Gamma$, that is, the set of limits points for the action of $\Gamma$ on $\overline\H$, $\Lambda(\Gamma)\subset\overline\R$. The limit set may be empty, or may contain one or two points, or may contain an infinite number. A Fuchsian group is of the first type if its limit set is the closed real line $\overline\R=\R\cup\{\infty\}$. This happens in the case of lattices, but there are Fuchsian groups of the first kind of infinite covolume. These latter are always infinitely generated.

When the limit set is finite, we say that $\Gamma$ is elementary. In such case, $\Gamma$ is cyclic.

In this work we shall mainly handle two type of Fuchsian groups. The first are Veech groups of Veech surfaces, which are lattices by definition and the other are the subgroups of the Veech group given by $\pker\rho_F$, for equivariant subbundles $F\subset H^1$. Recall that $\rho_F:\SL(X)\to\SL(F_X)$. Thus, $\ker\rho_F$ is a subgroup of $\SL(X)$, $\pker\rho_F$ is the image of $\ker\rho_F$ in $\PSL(X)$.

The following result allows us to better understand these groups when $F$ is a $2$-dimensional subbundle of $\ker\hol$.

\begin{theo}[{\cite[Theorem~5.6]{HoW}}] \label{theo:ker non-elementary}
Let $X$ be a Veech surface and $F$ an integer (defined over $\Z$) $2$-dimensional subbundle of $\ker\hol$ over the $\slr$-orbit of $X$. Then, $\pker \rho_F$ is a Fuchsian group of the first kind.
\end{theo}

In particular, in the case of Veech wind-tree billiards, the hypothesis are satisfied by the subbundles $F^{(h)}$ and $F^{(v)}$ and therefore, $\pker\rho_F$ is a Fuchsian group of the first kind for $F=F^{(h)},F^{(v)}$.

\subsubsection{Critical exponent} \label{sect:critical exponent}
Another concept which is of major relevance in this work is that of the critical exponent of a Fuchsian group.
For an introduction to the subject, we refer the reader to Peign\'e~\cite{Pe}.

Let $\Gamma$ be a Fuchsian group. The orbital function $n_\Gamma:\R_+\to\N$ is defined by $n_\Gamma(R)=\#\{g\in\Gamma:\; d_\H(i,gi)\leq R\}$.
The exponent \[\delta(\Gamma) \coloneqq \limsup_{R\to\infty} \frac{1}{R}\ln n_\Gamma(R)\] is the critical exponent of $\Gamma$. It corresponds to the critical exponent (the abscissa of convergence in $\R_+$) of the Poincar\'e series defined by \[P_\Gamma(s) \coloneqq \sum_{g\in\Gamma} e^{-s d_\H(i,gi)}.\]
That is, $P_\Gamma(s)$ diverges for $s < \delta(\Gamma)$ and converges for $s > \delta(\Gamma)$.

Note that in the definition of the critical exponent $\delta(\Gamma)$ it is innocuous if we change $d_\H(i,gi)$ for $d_\H(x,gy)$, for some $x,y\in\H$ or, in particular, if we change $\Gamma$ for some conjugate of $\Gamma$, either in the definition of the orbital function $n_\Gamma$ or in the Poincar\'e series $P_\Gamma$.

A result of Roblin~\cite{Ro1,Ro2} relates in a sharper way the asymptotic behavior of the orbital function and the critical exponent.

\begin{theo}[Roblin] \label{theo:orbital function}
Let $\Gamma$ be a non-elementary Fuchsian group. Then 
\[n_\Gamma(r) = O(e^{\delta r}),\] as $L\to\infty$.
\end{theo}

Consider now the following sungroups of $\pslr$:
\begin{itemize}
\item $\displaystyle K = \left\{\begin{pmatrix} \cos\theta & \sin\theta \\ -\sin\theta & \cos\theta \end{pmatrix}:\; \theta\in\left[0,\pi\right)\right\}$,
\item $\displaystyle A = \left\{\begin{pmatrix} e^t & 0 \\ 0 & e^{-t} \end{pmatrix}:\; t\in\R\right\}$, and
\item $\displaystyle N = \left\{\begin{pmatrix} 1 & t \\ 0 & 1 \end{pmatrix}:\; t\in\R\right\}$.
\end{itemize}
Every element $g\in\pslr\setminus\{id\}$ is conjugated to some element in $K$, $A$ or $N$. In fact, we have the following:
\begin{itemize}
\item $|\tr(g)|< 2$ if and only if $g$ is conjugated to some element of $K$. In this case $g$ is called elliptic and it fixes exactly one point in $\overline\H$, which belongs to $\H$;
\item $|\tr(g)|> 2$ if and only if $g$ is conjugated to some element of $A$. In this case $g$ is called hyperbolic and it fixes exactly two point in $\overline\H$, which belongs to $\partial\overline\H=\overline\R$; and
\item $|\tr(g)|= 2$ if and only if $g$ is conjugated to some (and therefore, to every) element of $N$. In this case $g$ is called parabolic and it fixes exactly one point in $\overline\H$, which belongs to $\partial\H$.
\end{itemize}

If $\Gamma$ is a non-elementary Fuchsian group, it has positive critical exponent $\delta(\Gamma)>0$ and if it contains a parabolic element, then $\delta(\Gamma) > 1/2$.

One of the main ingredients we use to prove our results is the following result of Brooks~\cite{Br} (see also~\cite{RT}).

\begin{theo}[Brooks] \label{theo:Brooks}
Let $\Gamma_0$ be a Fuchsian group and $\Gamma$ be a non-elementary normal subgroup of $\Gamma_0$ such that $\delta(\Gamma)>1/2$.
\begin{enumerate}
\item If $\rquo{\Gamma_0}{\Gamma}$ is amenable, then $\delta(\Gamma) = \delta(\Gamma_0)$.
\item If $\Gamma_0$ is a lattice and $\rquo{\Gamma_0}{\Gamma}$ is non-amenable, then $\delta(\Gamma) < \delta(\Gamma_0) = 1$.
\end{enumerate}
\end{theo}

This last result is based on the fact that the critical exponent $\delta(\Gamma)$ is related to $\lambda_0(\Gamma)$, the bottom of the spectrum of the Laplace operator on $\rquo{\H}{\Gamma}$. In fact, when $\delta(\Gamma)\geq1/2$, we have that $\lambda_0(\Gamma) = \delta(\Gamma)(1-\delta(\Gamma))$ (see for example~\cite{RT}).


\section{Counting problems on Veech surfaces} \label{sect:cpvs}

Let $X$ be a Veech surface, that is, $X$ is a translation surface whose Veech group $\PSL(X)$ is a (non-uniform) lattice. In particular, $\rquo{\H}{\PSL(X)}$ has a finite number of cusps. It is well known (since Veech~\cite{Ve1}) that, for Veech surfaces, cylinders correspond to the cusps of the Veech group and, in particular, the family of all cylinders can be written as the union of a finite number of $\SL(X)$-orbit of cylinders. That is, there are finitely many cylinders $A_1,\dots,A_n$ in $X$ such that
\[\A\coloneqq\{\text{all cylinders in } X\} = \SL(X)\cdot \{A_j\}_{j=1}^{n}.\]
In particular, any collection $\C\subset\A$ of cylinders is contained in a finite union of cusps, in the sense that it satisfies $\C \subset \SL(X)\cdot \mbf C$, for some \emph{finite} collection $\mbf C \subset\C$.

\subsection{Finitely saturated collections of cylinders} \label{sect:saturated}
Let $\Gamma$ be a 
subgroup of $\SL(X)$.
A 
collection $\C$ of cylinders in $X$ is said to be \emph{finitely saturated by $\Gamma$} (or \emph{$\Gamma$-finitely saturated}) if it can be expressed as a finite union of $\Gamma$-orbits of cylinders and $\Gamma$ contains every cusp. More precisely, $\C$ is finitely saturated by $\Gamma$ if $\C=\Gamma\cdot \mbf C$, for some finite collection $\mbf{C}\subset\C$ and $stab_{\SL(X)}(C) \subset \Gamma$ for every $C\in\C$. Equivalently, we can ask $stab_{\SL(X)}(C) \subset \Gamma$ only for $C\in\mbf{C}$.

Thus, as already said in different terms, the collection $\A$ of all cylinders in $X$ is $\SL(X)$-finitely saturated.

\begin{rema} In the definition of finitely saturated collections of cylinders, the \emph{finite} part is fundamental. Consider, for example, the group $\Gamma$ generated by all parabolics in $\SL(X)$. Then, when the Teichm\"uller curve defined by $X$ has positive genus\footnote{See~\cite{HL} for examples of Teichm\"uller curves with arbitrary large genus in a fixed stratum.}, $\A$ is saturated by $\Gamma$, but it is not $\Gamma$-finitely saturated.
\end{rema}

In general, any $\slr$-equivariant collection of cylinders (defined in the $\slr$-orbit of $X$) is $\SL(X)$-finitely saturated. In particular, configurations of cylinders, in the sense of Eskin--Masur--Zorich~\cite{EMZ}, define $\SL(X)$-finitely saturated collections of cylinders.
However, in this work, we have to deal with collections of cylinders which are finitely saturated by groups which are not lattices as $\SL(X)$ is. In fact, we have to deal with groups which are not even finitely generated.

\begin{rema}
If $\Gamma$ is a Fuchsian group such that a (non-empty) collection of cylinders $\C$ is finitely saturated by $\Gamma$, then, by definition, $stab_{\SL(X)}(C) \subset \Gamma$ for every $C\in \C$. But $\P stab_{\SL(X)}(C)$ is cyclic parabolic. Thus, $\Gamma$ contains parabolics and therefore $\delta(\Gamma)\geq 1/2$, with equality if and only if $\Gamma$ is elementary (and $\C$ is a finite collection of parallel cylinders).
\end{rema}

\subsection{Counting problem}

We are interested in counting cylinders in some particular collections. Let $\C$ be a collection of cylinders in $X$ and let $N_\C(X,L)$ be the number of cylinders in $C$ of length at most $L$.
We are able to study the asymptotic behavior in the case of finitely saturated collections.

In the case of $\A$, the collection of all cylinders in $X$, Veech proved the quadratic asymptotic behavior in~\cite{Ve1} and gave then an effective version in~\cite[Remark~1.12]{Ve2}.
In the case of collections of cylinders saturated by lattice groups, Veech's approach can be applied exactly the same. In fact, we have the following.

\begin{theo}[Veech] \label{theo:saturated Veech}
Let $X$ be a Veech surface and let $\C$ be a $\Gamma$-finitely saturated collection of cylinders on $X$ with $\Gamma$ being a lattice. Then
\[N_\C(X,L) = c(\C)L^2 + \sum_{j=1}^{k} c_j(\C)L^{2\delta_j} + O(L^{4/3}),\]
as $L\to\infty$, for some $c(\C),c_1(\C),\dots,c_k(\C)>0$, where $\{\delta_j(1-\delta_j)\}_{j=1}^{k}$ is the discrete spectrum of the Laplace operator on $\rquo{\H}{\Gamma}$ on $(0,1/4)$. In particular, $\delta_j\in(1/2,1)$, for $j=1,\dots,k$. Possibly $k=0$.
\end{theo}

\begin{proof}
For $\C=\A$, the collection of all cylinders in $X$ (which is finitely saturated by $\Gamma=\SL(X)$), Veech proved in \cite{Ve1} the principal term $c(\C)L^2$. The remainder was observed in \cite[Remark~1.12]{Ve2}, by an application of \cite[Theorem~4]{Go}.
The proof relies only in the fact that $\A$ is finitely saturated by a lattice group, namely $\SL(X)$.
Thus, in the case of collections finitely saturated by a lattice group, the proof follows exactly the same.
\end{proof}

In the case of infinite covolume groups this method cannot be adapted properly. However, following ideas of Dal'Bo~\cite{Da}, we are able to prove the following.

\begin{theo} \label{theo:critical exponent}
Let $X$ be a Veech surface and $\C$, a $\Gamma$-finitely saturated collection of cylinders on $X$ with $\Gamma$ non-elementary. Let $\delta=\delta(\Gamma)$ be the critical exponent of $\Gamma$. In particular, $\delta>1/2$.
Then, \[N_\C(X,L) = O(L^{2\delta}),\] as $L\to\infty$.
\end{theo}

\begin{proof}
Without loss of generality, we can assume that $\C = \Gamma\cdot C$, for some cylinder $C$ in $X$.
Let $p=\begin{psmallmatrix} 1 & 1 \\ 0 & 1 \end{psmallmatrix}$, $P=\langle p \rangle$ and $x=\begin{psmallmatrix} 0\\1 \end{psmallmatrix}$.
Up to conjugation, we can suppose that $\hol(\gamma_C)=x$ and $stab_{\SL(X)}(C)=P$. Note that $\delta$ is invariant by conjugation, so there is no loss of generality. Denote $N_\Gamma(L)\coloneqq N_{\Gamma\cdot C}(X,L)$.
The idea is to relate $N_\Gamma$ to $n_\Gamma$ in order to apply Theorem~\ref{theo:orbital function}.

It is clear that
\begin{align*}
N_\Gamma(L)
& = \#\{gx:\lvert gx \rvert \leq L,\, g\in\Gamma\} \\
& = \#\{gP\in\rquo{\Gamma}{P}:\lvert gx \rvert \leq L\} \\
& = \#\{Pg\in\lquo{\Gamma}{P}:\lvert g^{-1}x \rvert \leq L\}.
\end{align*}
A simple computation shows that $\lvert g^{-1}x \rvert = \Im(gi)^{-1/2}$. In addition, for each coset in $\lquo{\Gamma}{P}$, there is exactly one representative $g\in\Gamma$ such that $\Re(gi)\in \left[0,1\right)$. Thus,
\begin{align*}
N_\Gamma(L)
& = \#\{Pg\in\lquo{\Gamma}{P}:\lvert g^{-1}x \rvert \leq L\} \\
& = \#\{g\in\Gamma:\Re(gi)\in \left[0,1\right) ,\; \Im(gi)^{-1/2} \leq L\}.
\end{align*}

Moreover, there exists $c(\Gamma)>0$ such that if $g\in\Gamma$ satisfies $\Re(gi)\in \left[0,1\right)$, then $d_\H(i,gi)\leq -\ln \Im(gi) + c(\Gamma)$. In fact, let $g\in\Gamma$. Note first that $\Im(gi)$ is bounded above, since $P$ is a subgroup of $\Gamma$ (we have a cusp at infinity). 
In addition, we have that
\[d_\H(i,gi) = \acosh \left( 1 + \frac{\Re(gi)^2 + (1-\Im(gi))^2}{2\Im(gi)} \right)\]
and therefore, if $g\in\Gamma$ and $\Re(gi)\in \left[0,1\right)$, then
\[d_\H(i,gi) \leq \acosh \left( 1 + \frac{\tilde c(\Gamma)}{\Im(gi)} \right),\]
for some $\tilde c(\Gamma)>0$.
Once again, since $\Im(gi)$ is bounded above, we get that
\[d_\H(i,gi) \leq \ln \left(\frac{1}{\Im(gi)} \right) + c(\Gamma),\]
for some $c(\Gamma)>0$.

It follows that
\begin{align*}
N_\Gamma(L)
& = \#\{g\in\Gamma:\Re(gi)\in \left[0,1\right) ,\; \Im(gi)^{-1/2} \leq L\} \\
& \leq \#\{g\in\Gamma:d_\H(i,gi) \leq 2\ln L + c(\Gamma)\} \\
& = n_\Gamma(2\ln L + c(\Gamma)).
\end{align*}

Finally, by Theorem~\ref{theo:orbital function},
$n_{\Gamma}(r) = O(e^{\delta(\Gamma)r})$ and thus
\[
N_{\Gamma}(L)
\leq n_{\Gamma}(2\ln L + c(\Gamma)) \\
= O(e^{\delta(\Gamma)(2\ln L + c(\Gamma))}) \\
= O(L^{2\delta(\Gamma)}).
\qedhere
\]
\end{proof}

\section{Veech wind-tree billiards} \label{sect:vwtb}

In~\cite{Pa}, we proved asymptotic formulas for generic wind-tree models.
To prove such result, we had to split the associated collection of cylinders into two. The collection of \emph{good cylinders} and the collection of \emph{bad cylinders} (see \S\ref{sect:good bad cylinders}).
We proved then that good cylinders have quadratic asymptotic growth rate (and gave the associated coefficient in the generic case) and that bad cylinders have sub-quadratic asymptotic growth rate.

In this work we exhibit a quantitative version of these results in the case of Veech wind-tree billiards.

\subsection{Good cylinders}

Being a good cylinder is a $\slr$-invariant condition, then, in particular, for Veech wind-tree billiards $\Pi$, with Veech group $\PSL(\Pi)$ (see \S\ref{sect:vwtm}), the collection of good cylinders is $\SL(\Pi)$-finitely saturated (see \S\ref{sect:saturated}) and thus, as a corollary of Veech's theorem (Theorem~\ref{theo:saturated Veech}), we obtain the following. 

\begin{coro} \label{coro:good cylinders}
Let \,$\Pi$ be a Veech wind-tree billiard. Then, there exists $c(\Pi)>0$ and $\delta_{good}(\Pi)\in\left[1/2,1\right)$ such that \[ N_{good}(\Pi,L) = c(\Pi) \cdot \frac{\pi L^2}{\operatorname{Area}\left(\Pi/\Z^2\right)} + O(L^{2\delta_{good}(\Pi)})+ O(L^{4/3})\]
as $L\to\infty$, where $\delta=\delta_{good}(\Pi)$ is such that $\delta(1-\delta)$ is the second smallest eigenvalue of the Laplace operator on $\rquo{\H}{\PSL(\Pi)}$, $\delta(1-\delta)\in(0,1/4]$.
\end{coro}

\subsection{Bad cylinders} \label{sect:bad cylinders}

In the case of bad cylinders, Veech's approach is no longer possible since collection of bad cylinders is not $\slr$-equivariant and, in particular, bad cylinders are not $\SL(\Pi)$-finitely saturated. However, it is finitely saturated by a subgroup $\Gamma_{bad}$ of $\SL(\Pi)$, so we can use the approach on Theorem~\ref{theo:critical exponent}.

\begin{rema}\label{rema:infinitely generated}
We shall see that $\Gamma_{bad}$ is quite intricate. It is a not normal subgroup of $\SL(\Pi)$ and it is an infinitely generated Fuchsian group of the first kind.
\end{rema}

By this means, we prove that bad cylinders have sub-quadratic asymptotic growth rate in an effective way.
More precisely, we prove the following.

\begin{theo} \label{theo:bad cylinders}
Let \,$\Pi$ be a Veech wind-tree billiard. Then, there exists $\delta_{bad}(\Pi)\in\left(1/2,1\right)$ such that \[ N_{bad}(\Pi,L) = O(L^{2\delta_{bad}(\Pi)})\]
as $L\to\infty$.
\end{theo}

\begin{proof}

Let $f=h,v$ and $F=F^{(f)}$. Henceforth, by bad cylinder we mean $(F,f)$-bad cylinder.
Recall that a cylinder $C$ in $X=X(\Pi)$ is a bad cylinder if and only if $\pr_F\gamma_C=\pm f$ (see Remark~\ref{rema:cylinder projections}).

Let $\mathcal{B}$ be the collection of all bad cylinders in $X$. Then, since the collection of all cylinders can be written as a finite union of $\SL(X)$-orbits of cylinders, then there is a finite collection of bad cylinders $\mbf{B}$ such that $\mathcal{B} \subset \SL(X)\cdot \mbf{B}$.

Now, given a bad cylinder $B$ in $X$, define \[\Gamma_{bad}(B)\coloneqq\{g\in\SL(X): g\cdot B \text{ is a bad cylinder}\},\] so that \[\B=\bigcup_{B\in\mbf{B}} \Gamma_{bad}(B)\cdot B.\]

Since $B$ is a bad cylinder if and only if $\pr_F\gamma_C=\pm f$, then $g\in\Gamma_{bad}(B)$ if and only if $\pr_F\gamma_{g\cdot B} = \pm f$.
But $\pr_F\gamma_{g\cdot B} = \pr_F \rho_{H^1}(g)\gamma_B = \rho_F(g) \pr_F \gamma_B = \rho_F(g)(\pm f)$, where $\rho_F$ denotes the representation of $\SL(X)$ on $\Sp(F_X,\Z)$ (see \S\ref{sect:representation} and \S\ref{sect:wind-tree subbundles}).
It follows then that \[\Gamma_{bad}(B)=\Gamma_{bad}\coloneqq\{g\in\SL(X): \rho_F(g) f = \pm f\},\] which is a group and does not depend on $B\in\B$. Thus, $\mathcal{B} = \Gamma_{bad} \cdot \mbf{B}$.
Moreover, if $B\in\mathcal{B}$ and $p\in stab_{\SL(X)}(B)$, then $p\cdot B = B$, which is a bad cylinder. Therefore, $p\in \Gamma_{bad}(B) = \Gamma_{bad}$ and $\mathcal{B}$ is finitely saturated by $\Gamma_{bad}$ (see \S\ref{sect:saturated}).

We can apply then Theorem~\ref{theo:critical exponent} to obtain $\delta_{bad}(\Pi)=\delta(\Gamma_{bad})$. To conclude, we have to prove that $\delta(\Gamma_{bad})<1$. In fact, we have the following result, whose proof is postponed to \S\ref{sect:proof Gamma bad}.

\begin{prop} \label{prop:Gamma bad}
The critical exponent of $\Gamma_{bad}$ is strictly less than one.
\end{prop}

Thus, by Proposition~\ref{prop:Gamma bad} and Theorem~\ref{theo:critical exponent}, $N_{bad}(\Pi,L) = O(L^{2\delta_{bad}(\Pi)})$ as $L\to\infty$, where $\delta_{bad}(\Pi) = \delta(\Gamma_{bad})\in\left(1/2,1\right)$. This proves Theorem~\ref{theo:bad cylinders}.
\end{proof}

To conclude, we have to prove now Proposition~\ref{prop:Gamma bad}.

\subsubsection{Proof of Proposition~\ref{prop:Gamma bad}} \label{sect:proof Gamma bad}

Consider the normal subgroup of $\SL(X)$ given by $\ker \rho_F$ and note that it is also a subgroup of $\Gamma_{bad}$.

Since the action on homology is via (symplectic) integer matrices, then
\[\rho_F(\Gamma_{bad}) \subset stab(\pm f)\coloneqq \{ \hat g \in\Sp(F_X,\Z) : \hat g f = \pm f \}.\]
Since $F_X$ is two-dimensional, $\Sp(F_X,\Z)\cong \slz$ and $stab(\pm f)\cong stab_\slz(\pm\begin{psmallmatrix} 0\\1 \end{psmallmatrix})$, which is virtually cyclic parabolic. Thus, the quotient group $\rquo{\Gamma_{bad}}{\ker \rho_F}\cong\rho_F(\Gamma_{bad})$ is amenable (as it is isomorphic to a subgroup of an amenable group).

In a slight abuse of notation we will refer in the following to (discrete) subgroups of $\slr$ as if they were Fuchsian groups (discrete subgroups of $\pslr$).

By Theorem~\ref{theo:ker non-elementary}, $\ker \rho_F$ is of the first kind and, in particular, non-elementary.
Thus, we can apply Theorem~\ref{theo:Brooks} to obtain that $\delta(\Gamma_{bad}) = \delta(\ker \rho_F)$.

Consider now the quotient group $\rquo{\SL(X)}{\ker \rho_F} \cong \rho_F(\SL(X))$. The aim is to prove that $\rho_F(\SL(X))$ is not amenable.
We first note that, since $F$ has positive Lyapunov exponents (Theorem~\ref{theo:positive Lyapunov exponent}), $\rho_F(\SL(X))$ has at least one hyperbolic element and then, a maximal cyclic hyperbolic subgroup $H$. Suppose $\rho_F(\SL(X))$ is elementary and, in particular, virtually $H$. But then, $F$ would admit an almost invariant splitting (see \S\ref{sect:subbundles}).
But $F$ is two dimensional and has no zero Lyapunov exponents, in particular, it is strongly irreducible and do not admit almost invariant splittings. Thus $\rho_F(\SL(X))$ is non-elementary and it contains a Schottky group as subgroup.

Since Schottky groups are free and, in particular, non-amenable, it follows that $\rho_F(\SL(X))$ is not amenable. That is, $\rquo{\SL(X)}{\ker \rho_F}$ is not amenable, and then, by Theorem~\ref{theo:Brooks}, we have that $\delta(\ker \rho_F) < \delta(\SL(X))$. Thus, we conclude that \[\delta(\Gamma_{bad})=\delta(\ker \rho_F) < \delta(\SL(X))=1.\pushQED{\qed}\qedhere\]

\begin{proof}[Proof of Remark~\ref{rema:infinitely generated}]
We have to show that $\Gamma_{bad}$ is an infinitely generated group of the first kind.
Since $\ker\rho_F$ is of the first kind and $\ker\rho_F\subset \Gamma_{bad}$, so is $\Gamma_{bad}$. Moreover, $\delta(\Gamma_{bad})<1$, so it cannot be a lattice and therefore, it has to be infinitely generated, since finitely generated groups of the first kind are always lattices.
\end{proof}

\section{Explicit estimates for the $(1/2,1/2)$ wind-tree model} \label{sect:estimates}

In the case of the wind-tree billiard with square obstacles of side length $1/2$, $\Pi=\Pi(1/2,1/2)$, the Veech group can be easily computed (see \S\ref{sect:example}). Indeed, $\SL(\Pi)=\langle u^2,{}^{t\!}u^2 \rangle$, where $u= \begin{psmallmatrix} 1 & 1 \\ 0 & 1 \end{psmallmatrix}$.
In particular, $\PSL(\Pi)$ is a congruence subgroup of level~$2$.

\subsection{Good cylinders} \label{sect:good estimates}
A result of Huxley~\cite{Hu} shows that congruence groups $\Gamma$ of low level satisfies the Selberg's~$1/4$ conjecture, that is, that the spectral gap of the Laplace operator on $\rquo{\H}{\Gamma}$ equals $1/4$. That means (see 
\S\ref{sect:spectral gap}) that we have $\delta_{good}(\Pi)=1/2$ in Corollary~\ref{coro:good cylinders}.

\subsection{Bad cylinders}
We have now to estimate $\delta_{bad}(\Pi)$ from Theorem~\ref{theo:bad cylinders}. For this, we use a version of Brook's theorem (Theorem~\ref{theo:Brooks}) by Roblin--Tapie~\cite{RT}, formulated in a much more general context, which we adapt to ours.

\begin{theo}[Roblin--Tapie] \label{theo:RT}
Let $\Gamma_0$ be a lattice and $\Gamma$ be a non-elementary normal subgroup of $\Gamma_0$ such that $\delta(\Gamma)>1/2$.
Let $\D$ be a Dirichlet domain for $\Gamma_0$ and $S_0$ the associated symmetric system of generators (see \S\ref{sect:transition zones}).
Consider $G=\rquo{\Gamma_0}{\Gamma}$ and $S=\rquo{S_0}{\Gamma}$ the corresponding systems of generators of $G$. Then,
\[\lambda_0(\Gamma) \geq \frac{\eta(\Gamma_0) E_\D \mu_0(G,S)}{\eta(\Gamma_0) + E_\D \mu_0(G,S)},\]
where $\eta(\Gamma_0)$ is the spectral gap associated to $\Gamma_0$ (see \S\ref{sect:spectral gap}), $E_\D$ is any lower bound for the energy on $\D$ (see \S\ref{sect:energy}) and $\mu_0(G,S)$ is the bottom of the combinatorial spectrum of $G$ associated to $S$ (see \S\ref{sect:combinatorial spectrum}), as defined below.
\end{theo}

\subsubsection{Critical exponent and spectrum of the Laplace operator} \label{sect:spectral gap}
Let $\Gamma$ be a non-elementary Fuchsian group with critical exponent $\delta(\Gamma)>1/2$.
Then, the critical exponent $\delta(\Gamma)$ is related to $\lambda_0(\Gamma)$, the bottom of the spectrum of the Laplace operator on $\rquo{\H}{\Gamma}$, by $\lambda_0(\Gamma) = \delta(\Gamma)(1-\delta(\Gamma))\in\left(0,1/4\right)$.

If moreover $\Gamma$ is finitely generated, then the bottom of the spectrum $\lambda_0(\Gamma)$ is an isolated eigenvalue. We consider then the \emph{spectral gap of the Laplace operator on $\rquo{\H}{\Gamma}$}, $\eta(\Gamma) \coloneqq \lambda_1(\Gamma) - \lambda_0(\Gamma)>0$, where $\lambda_1(\Gamma)$ is the second smallest eigenvalue of the Laplace operator on $\rquo{\H}{\Gamma}$.

\subsubsection{Dirichlet domains and transition zones} \label{sect:transition zones}
Let $\Gamma$ be a finitely generated Fuchsian group and consider a Dirichlet domain $\D\subset\H$ for the action of $\Gamma$. Its boundary $\partial\D$ is piecewise geodesic, with finitely many pieces. To $\D$, we can associate a finite symmetric system of generators $S$ of $\Gamma$. Each such generator $s\in S$ is associated to one geodesic piece of $\partial\D$. Namely, $\beta_s=\D\cap s\D$. And every geodesic piece of $\partial\D$ has an associated generator in this way. 


%

We say that $L,R>0$ are \emph{admisible (for $\D$)} if for each $s\in S$, there exists a geodesic segment $\alpha_s\subset\beta_s$ of length $L$ such that $\alpha_s = s\alpha_{s^{-1}}$ and such that $\alpha_s$ admits a tubular neighborhood of radius $R$ which are pairwise disjoint (see Appendix~\ref{sect:estimates energy} for more details). These tubular neighborhoods are \emph{transition zones} of length $L$ and radius $R$ (cf.~\cite[p.~72]{RT}).



\subsubsection{Energy on transition zones} \label{sect:energy}
Roblin--Tapie~\cite{RT} introduced the volume and capacity of transition zones (in a much more general context). In our context, for a transition zone of length $L$ and radius $R$, its \emph{area} is $\Vol(L,R) \coloneqq L\cdot \sinh(R)$ and its \emph{capacity} is $\Cap(L,R) \coloneqq L/\arctan(\sinh(R))$.
We say that $E_\D\in\R_+$ is a \emph{lower bound for the energy on $\D$} if there are admissible $L,R>0$ such that
\[ E_\D =\frac{1}{2 \operatorname{Area}(\D)}\cdot\frac{\eta(\Gamma) \cdot \Vol(L,R) \cdot \Cap(L,R)}{\left(\sqrt{\eta(\Gamma) \cdot \Vol(L,R)} + \sqrt{\Cap(L,R)} \right)^2}.\]

In Appendix~\ref{sect:estimates energy} we estimate $E_\D$ in the case of the Dirichlet domain of the Veech group of $\Pi$, $\D = \left\{ z \in \H: |z\pm 1/2| \geq 1/2,\, |\Re(z)| \leq 1 \right\}$, with associated system of generators $S_0=\{u^2,{}^{t\!}u^2\}$.

\subsubsection{Combinatorial spectrum} \label{sect:combinatorial spectrum}
Let $G$ be a finitely generated group and $S\subset G$ be a symmetric finite system of generators of $G$.

Let $\ell^2(G)$ be the space of square-summable sequences on $G$ with the inner product \[\langle h,h' \rangle \coloneqq \sum_{g\in G} h_g {h\mathrlap{'}}_g,\] for $h,h'\in\ell^2(G)$, and define $\Delta_S:\ell^2(G)\to\ell^2(G)$, the combinatorial Laplace operator associated to $S$ on $\ell^2(G)$, by \[(\Delta_S h)_g \coloneqq \sum_{s\in S} (h_g - h_{gs}).\]

Then, we define $\mu_0(G,S)$, \emph{the bottom of the combinatorial spectrum of $G$ associated to $S$} to be the bottom of the spectrum of $\Delta_S$, that is, \[\mu_0(G,S)\coloneqq \inf \left\{\frac{\langle \Delta_S h, h \rangle}{\langle h, h \rangle},\; h\in \ell^2(G)\right\}.\]

We estimate $\mu_0(G,S)$ in the case of $G<\pslz$ generated by $\{u,{}^{t\!}u^3\}$ in Appendix~\ref{sect:estimates combinatorial spectrum}.

\subsection*{Estimates for $\delta_{bad}(\Pi)$}

An application of Theorem~\ref{theo:RT} allows us to estimate $\delta_{bad}(\Pi)$ in the present case. More precisely, we have the following.

\begin{theo} \label{theo:estimates}
Let \,$\Pi$ be the Veech wind-tree billiard with square obstacles of side length $1/2$, and let $\delta=\delta_{bad}(\Pi)\in\left(1/2,1\right)$ be as in the conclusion of Theorem~\ref{theo:bad cylinders}. Then, \[\delta < 0.9885.\]
\end{theo}

\begin{proof}
Following \S\ref{sect:bad cylinders}, we have that $\delta=\delta_{bad}(\Pi)$ corresponds to the critical exponent of the group $\Gamma_{bad}$\footnote{Here, in a slight abuse of notation, we are referring to (discrete) subgroups of $\slr$ as if they were Fuchsian groups (discrete subgroups of $\pslr$).}. Moreover, $\delta(\Gamma_{bad}) = \delta(\ker \rho_F)$. Let then $\delta = \delta(\ker \rho_F)$. As we have already seen, $\delta(1-\delta) = \lambda_0(\ker \rho_F)$.

The idea is to apply Theorem~\ref{theo:RT} to $\Gamma_0=\PSL(\Pi)$ and $\Gamma = \pker \rho_F$. Thus, it is enough to estimate \[\frac{\eta(\Gamma_0) E_\D \mu_0(G,S)}{\eta(\Gamma_0) + E_\D \mu_0(G,S)}\] from below.

Note that the function $x/(1+x)$ is an increasing function in $\left(0,\infty\right)$ and therefore, the problem can be reduced to find lower bounds for $\eta(\Gamma_0)$, $E_\D$ and $\mu_0(G,S)$.
\begin{itemize}
\item $\Gamma_0=\langle u^2, {}^{t\!}u^2 \rangle$ is a level two congruence group and, as already seen in \S\ref{sect:good estimates}, its spectral gap is \[\eta(\Gamma_0) = 1/4.\]
\item We consider $\D = \left\{ z \in \H: |z\pm 1/2| \geq 1/2,\, |\Re(z)| \leq 1 \right\}$, the Dirichlet domain for $\Gamma_0$. We estimate $E_\D$ in Appendix~\ref{sect:estimates energy}. By Theorem~\ref{theo:energy}, we have that \[E_\D > 0.02575.\]
\item Recall that $\SL(\Pi)=\langle u^2, {}^{t\!}u^2 \rangle$ and that, $\rho_{F^{(h)}}(\SL(\Pi))$,  $\rho_{F^{(v)}}(\SL(\Pi))$ are conjugated to $\langle u, {}^{t\!}u^3 \rangle$ (see \S\ref{sect:example}).

Moreover since $F=F^{(h)}$ or $F^{(v)}$ is a $2$-dimensional symplectic equivariant subbundle defined over $\Z$, $\rho_F$ descends to a representation $\tilde \rho_F$ of $\PSL(X)$ on $\PSL(F_X,\Z)$ (see \S\ref{sect:representation}), where $X=X(\Pi)$. Furthermore, by definition, the kernel of this latter representation coincides with $\pker \rho_F$, the image of $\ker \rho_F$ in $\pslr$. Analogously, for the image of the representation we have $\tilde\rho_F(\PSL(X))=\P\rho_F(\SL(X))$.
In summary, we have
\begin{itemize}
\item $\Gamma_0=\PSL(\Pi)=\langle u^2, {}^{t\!}u^2 \rangle$,
\item $\Gamma = \pker\rho_F = \ker\tilde\rho_F$,
\item $\rquo{\Gamma_0}{\Gamma}=\rquo{\PSL(\Pi)}{\ker\tilde\rho_F} \cong \tilde \rho_F(\PSL(X))=\P\rho_F(\SL(X))$, and
\item $\rho_F(\SL(X)) \cong \langle u, {}^{t\!}u^3 \rangle \eqqcolon H_3$.
\end{itemize}

The combinatorial spectrum is invariant under isomorphisms of groups (with generators). But $\rquo{\Gamma_0}{\Gamma}$ is isomorphic to $\P\rho_F(\SL(X))$ which in turn is isomorphic to $\P H_3$. In addition, the system of generators associated to the Dirichlet domain $\D$ is $S_0=\{u^{\pm2}, {}^{t\!}u^{\pm2}\}$, and the corresponding image into $G=\P H_3$ is $S=\{u^{\pm1},{}^{t\!}u^{\pm3}\}$.

We estimate $\mu_0(G,S)$ in Appendix~\ref{sect:estimates combinatorial spectrum}. By Theorem~\ref{theo:combinatorial spectrum}, we have that \[\mu_0(G,S) > 0.4647.\]
\end{itemize}

Putting all together, we get that
\[\lambda_0(\Gamma) \geq \frac{\eta(\Gamma_0) E_\D \mu_0(G,S)}{\eta(\Gamma_0) + E_\D \mu_0(G,S)} > 0.01141,\]
and we conclude that
\[\delta(\Gamma) = \frac{1+\sqrt{1-4\lambda_0(\Gamma)}}{2} < 0.9885.\qedhere\]
\end{proof}

\appendix
\renewcommand{\thefigure}{\thesection.\arabic{figure}}
\renewcommand{\theHfigure}{\thesection.\thefigure}

\setcounter{figure}{0}
\section{Energy estimates}
\label{sect:estimates energy}

In this appendix we give lower bounds for the energy (see \S\ref{sect:energy2} for precise definition) on the Dirichlet domain $\D = \left\{ z \in \H: |z\pm 1/2| \geq 1/2,\, |\Re(z)| \leq 1 \right\}$ of the Fuchsian group $\Gamma=\langle u^2, {}^{t\!}u^2 \rangle$, where $u=\begin{psmallmatrix} 1 & 1 \\ 0 & 1 \end{psmallmatrix}$ and ${}^{t\!}u$, its transpose. More precisely, we prove the following.

\begin{theo} \label{theo:energy}
Let $\D = \left\{ z \in \H: |z\pm 1/2| \geq 1/2,\, |\Re(z)| \leq 1 \right\}$ be the Dirichlet domain of $\Gamma=\langle u^2, {}^{t\!}u^2 \rangle$. Then, there is a lower bound for the energy on $\D$ which satisfies \[E_\D > 0.02575.\]
\end{theo}

In the following we recall the definition of the involved objects (see~\cite{RT} for a much more general and detailed discussion).

\subsection{Dirichlet domains and transition zones} \label{sect:transition zones2}
Let $\Gamma$ be a finitely generated Fuchsian group and consider a Dirichlet domain $\D\subset\H$ for the action of $\Gamma$. Its boundary $\partial\D$ is piecewise geodesic, with finitely many pieces. To $\D$, we can associate a finite symmetric system of generators $S$ of $\Gamma$. To each such generator $s\in S$ we can associate one geodesic piece of $\partial\D$. Namely, $\beta_s=\D\cap s\D$. And every such piece has an associated generator in this way. Moreover, it is clear from the definition that $\beta_s = s\beta_{s^{-1}}$. In Figure~\ref{figu:Dirichlet domain2}, we show the case of the elementary group $\langle u \rangle$.

\begin{figure}[ht]
\centering\includegraphics[width=.9\textwidth,height=.25\textheight,keepaspectratio]{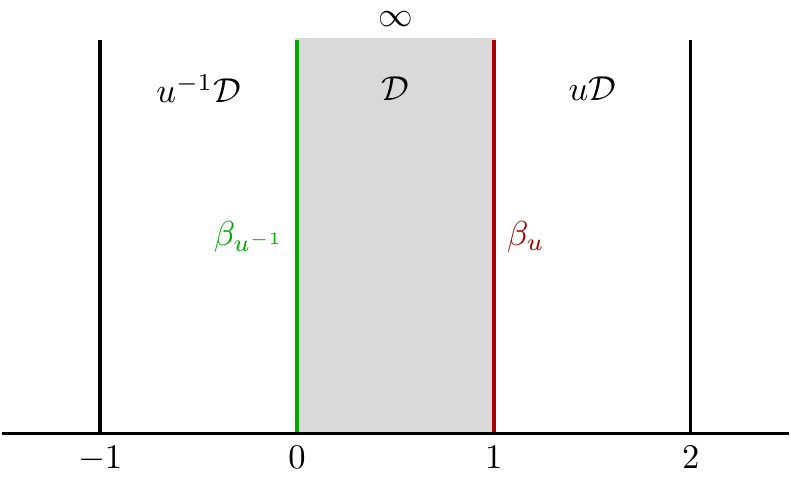}
\caption{Dirichlet domain for the elementary (cyclic parabolic) group $\langle u \rangle$, $\D=\left\{0\leq\Re z\leq 1\right\}$. The associated symmetric systems of generators is $S=\{u,u^{-1}\}$ and the corresponding geodesic boundaries $\beta_u=\left\{\Re z= 1\right\}$, $\beta_{u^{-1}}=\left\{\Re z= 0\right\}$.}
\label{figu:Dirichlet domain2}
\end{figure}

Let $z\in\mathring{\beta_s}$, for some $s\in S$, and let $\rho>0$ sufficiently small such that there is a point $b_s(z,\rho)\in\D$ satisfying $d_\H(b_s(z,\rho),\beta_s) = d_\H(b_s(z,\rho),z)= \rho$. In particular, such point $b_s(z,\rho)$ is unique.
See Figure~\ref{figu:b2} for an example of $b_s(z,\rho)$, in the case of $\langle u \rangle$, for $s=u^{-1}$.

\begin{figure}[ht]
\centering\includegraphics[width=.9\textwidth,height=.25\textheight,keepaspectratio]{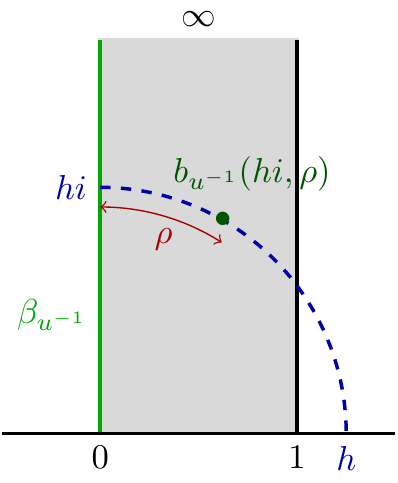}
\caption{The point $b_s(z,\rho)$. It corresponds to the point in $\D$ which lie on the geodesic passing through $hi$ perpendicularly to $\beta_{u^{-1}}$, in the case of the elementary group $\langle u \rangle$, for $s=u^{-1}$, $z=hi$}
\label{figu:b2}
\end{figure}

We say that $L,R>0$ are \emph{admisible (for $\D$)} if for each $s\in S$, there exists a geodesic segment $\alpha_s\subset\beta_s$ of length $L$ such that $\alpha_s = s\alpha_{s^{-1}}$, $b_s(z,R)$ is well defined and the sets
\[A_s\coloneqq \{b_s(z,\rho)\in\D :\; z\in\alpha_s,\; 0\leq \rho < R\}\] are pairwise disjoint. (see Figure~\ref{figu:transition zones2}). We call these sets, \emph{transition zones} of length $L$ and radius $R$ (cf.~\cite[p.~72]{RT}).

\begin{figure}[ht]
\centering\includegraphics[width=.9\textwidth,height=.25\textheight,keepaspectratio]{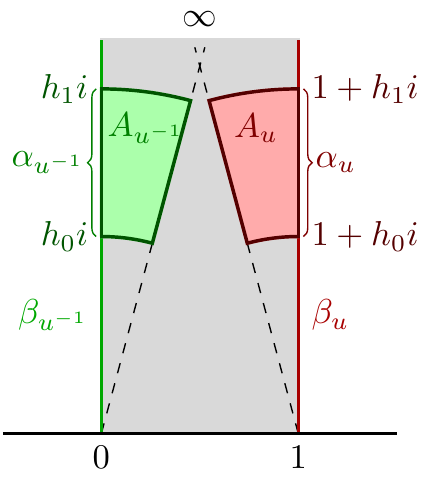}
\caption{Transition zones, in the case of the elementary group $\langle u \rangle$.}
\label{figu:transition zones2}
\end{figure}

\subsection{Energy on transition zones} \label{sect:energy2}
Roblin--Tapie~\cite{RT} introduced the volume and capacity of transition zones (in a much more general context). In our context, for a transition zone of length $L$ and radius $R$, its \emph{area} is $\Vol(L,R) \coloneqq L\cdot \sinh(R)$ and its \emph{capacity} is $\Cap(L,R) \coloneqq L/\arctan(\sinh(R))$.

We say that $E_\D\in\R_+$ is a \emph{lower bound for the energy on $\D$} if there are admissible $L,R>0$ such that
\[ E_\D =\frac{1}{2 \operatorname{Area}(\D)}\cdot\frac{\eta(\Gamma) \cdot \Vol(L,R) \cdot \Cap(L,R)}{\left(\sqrt{\eta(\Gamma) \cdot \Vol(L,R)} + \sqrt{\Cap(L,R)} \right)^2},\]
where $\eta(\Gamma)$ is the spectral gap of the Laplace operator on $\rquo{\H}{\Gamma}$, which is well defined and positive, since $\Gamma$ is finitely generated.

We can now start the discussion in our particular case, that is, the Dirichlet domain of $\Gamma=\langle u^2, {}^{t\!}u^2 \rangle$, $\D = \left\{ z \in \H: |z\pm 1/2| \geq 1/2,\, |\Re(z)| \leq 1 \right\}$.

\subsection{Proof of Theorem~\ref{theo:energy}}

The following result, whose proof is postponed to \S\ref{sect:admissible}, provides a sufficient condition for $L,R>0$ to be admissible (see \S\ref{sect:transition zones2}).

\begin{prop} \label{prop:admissible}
Let $L,R>0$. If $2e^L \tanh^3(R) \leq 1$, then $L,R$ are admissible.
\end{prop}

We want now to estimate $E_\D$ (see \S\ref{sect:energy2}).

We first note that $\Gamma=\langle u^2, {}^{t\!}u^2 \rangle$ is a congruence group of level two and therefore, by a result of Huxley~\cite{Hu}, we have that $\eta(\Gamma)=1/2$.
Moreover, the Dirichlet domain $\D$ is an ideal quadrilateral, with vertices $1$, $-1$, $0$ and $\infty$ (see Figure~\ref{figu:ideal quadrilateral}). In particular, $\operatorname{Area}(\D) = 2\pi$.

\begin{figure}[ht]
\centering\includegraphics[width=.9\textwidth,height=.25\textheight,keepaspectratio]{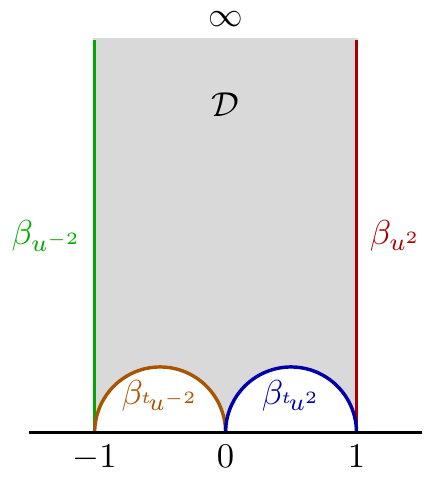}
\caption{Dirichlet domain for $\Gamma=\langle u^2, {}^{t\!}u^2 \rangle$. 
The associated symmetric systems of generators is $S=\{u^2, u^{-2}, {}^{t\!}u^2, {}^{t\!}u^{-2}\}$ and the corresponding geodesic boundaries are $\beta_{u^{\pm2}}=\left\{\Re z= \pm1\right\}$, $\beta_{{}^{t\!}u^{\pm2}}=\left\{|z\pm1/2|= 1/2\right\}$.}
\label{figu:ideal quadrilateral}
\end{figure}

By Proposition~\ref{prop:admissible}, $L,R>0$ are admissible if $2e^L \tanh^3(R) \leq 1$.
It suffices then to find the largest possible lower bound for the energy in this region. That is, we want to find $E^*=\max\{ E_\D(L,R):\,2 e^L \tanh^3 R \leq 1\}$. This can be done numerically: we get $L^*\approx 3.8903$, $R^*\approx 0.2205$ and \[E^* = E_\D(L^*,R^*) \approx 0.0257532 > 0.02575.\pushQED{\qed}\qedhere\]

\subsection{Proof of Proposition~\ref{prop:admissible}} \label{sect:admissible}
In this section we prove Proposition~\ref{prop:admissible}, thus providing a sufficient condition for $L,R>0$ to be admissible.

For $a,b\in\overline\R$, let $\gamma(a,b)$ denote the (bi-infinite) geodesic in $\H$ which goes from $a$ to $b$.
And for $x,y,z\in\overline\H$, let $T(x,y,z)$ denote the geodesic triangle with vertices $x,y,z$.
Thus, the Dirichlet domain $\D = \left\{ z \in \H: |z\pm1/2| \geq 1/2,\, |\Re(z)| \leq 1 \right\}$ coincides with $T(-1,1,\infty)\cup T(-1,1,0)$ and $\partial\D=\gamma(-1,0)\cup\gamma(0,1)\cup\gamma(1,\infty)\cup\gamma(\infty,-1)$ (see Figure~\ref{figu:ideal quadrilateral}).
Note that the symmetric system of generators associated to $\D$ is $S=\{u^{\pm2}, {}^{t\!}u^{\pm2}\}$ and, following the notation on \S\ref{sect:transition zones2}, we have
\begin{align*}
\beta_{u^{\pm2}} & =\{z\in \H:\;\Re(z) = \pm1\} = \gamma(\pm1,\infty), \\
\beta_{{}^{t\!}u^{\pm2}} & =\left\{|z\pm1/2|= 1/2\right\} = \gamma(\pm1,0).
\end{align*}
It follows that, in particular, any geodesic segment $\alpha_{u^{\pm2}}\subset\beta_{u^{\pm2}}$ (see \S\ref{sect:transition zones2}) is of the form $\{z\in \H:\;\Re(z) = \pm1,\; h_0 < \Im(z) < h_1\}$, for some $h_1>h_0>0$, with the same $h_1$ and $h_0$ for both $\alpha_{u^2}$ and $\alpha_{u^{-2}}$ since $\alpha_{u^2}=u^2\alpha_{u^{-2}}$. In such case, the length of the geodesic segment $\alpha_{u^{\pm2}}$ is equals to $L=log(h_1/h_0)$.

For simplicity, we shall consider a ``symmetric'' partition of $\D$ as in Figure~\ref{figu:partition}, given by a homography $g$, defined by the elliptic element $g=\begin{psmallmatrix} 0 & -1 \\ 1 & 0 \end{psmallmatrix}$, which is an isometry of order~$2$ fixing $i$ and such that permutes $-1$ with $1$ and $0$ with $\infty$. In particular, divides $\D$ in four isometric triangular regions. Namely, $T(-1,i,0)$, $T(0,i,1)$, $T(1,i,\infty)$ and $T(\infty,i,-1)$. Moreover, it is clear that
\begin{align*}
g: \quad
T(-1,i,0) & \leftrightarrow T(1,i,\infty), \\
T(\infty,i,-1) & \leftrightarrow T(0,i,1).
\end{align*}

\begin{figure}[ht]
\centering\includegraphics[width=.9\textwidth,height=.25\textheight,keepaspectratio]{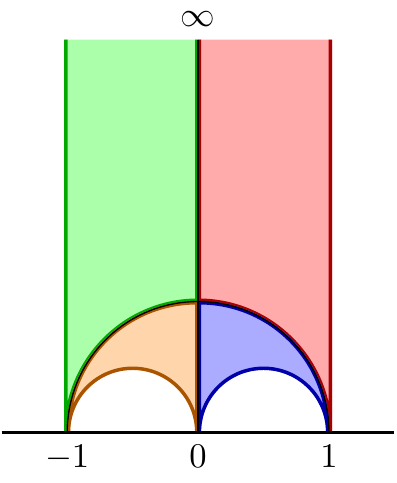}
\caption{Symmetric partition of $\D$}
\label{figu:partition}
\end{figure}

In particular, if we consider the transition zones to be contained in these triangular regions, it is direct that they are pairwise disjoint. And since these regions are isometric, we can consider the transition zones to be isometric and interchanged by the isometry $g$. That is, we impose
\begin{align*}
g: \quad
\alpha_{u^{-2}} & \leftrightarrow \alpha_{{}^{t\!}u^2}, \\
\alpha_{{}^{t\!}u^{-2}} & \leftrightarrow \alpha_{u^2}=u^2\alpha_{u^{-2}}.
\end{align*}

We have now to study the points $b_s(z,R)$, $s\in S$, $z\in\alpha_s$, in order to give conditions to $L,R$ to be admissible (see \S\ref{sect:transition zones2}). Moreover, by the imposed symmetries, it is enough to find conditions for $b_0(h,R)\coloneqq b_{u^2}(1+hi,R)$, to be contained in $T_0\coloneqq T(1,i,\infty)$, for $h>0$.

Now, by definition, $b_0(h,R)$ is the only point in $\D$ such that \[d_\H(b_0(h,R),\beta_{u^2}) = d_\H(b_0(h,R),1+hi)= R,\] for $R>0$ small enough.
By the leftmost equality, such points correspond to points in $\D$ which lie on the geodesic passing through $1+hi$ perpendicularly to $\beta_{u^2}$ (see Figure~\ref{figu:symmetric transition zones}, cf. Figure~\ref{figu:b2}). That is, $b_0(h,R) = 1 + he^{i\theta(R)}i$, for some $\theta(R)>0$.
Moreover,
\[
d_\H(1 + he^{i\theta}i,1+hi)
 = d_\H(e^{i\theta}i,i)
 = \acosh(\sec(\theta)).
\]
Thus, $\cos(\theta(R)) = \sech(R)$ and therefore, $\sin(\theta(R)) = \tanh(R)$. It follows that
\[b_{u^2}(1+hi,R) = b_0(h,R) = 1 + he^{i\theta(R)}i = 1 - h\tanh(R) + ih\sech(R).\]

\begin{figure}[ht]
\centering\includegraphics[width=.9\textwidth,height=.29\textheight,keepaspectratio]{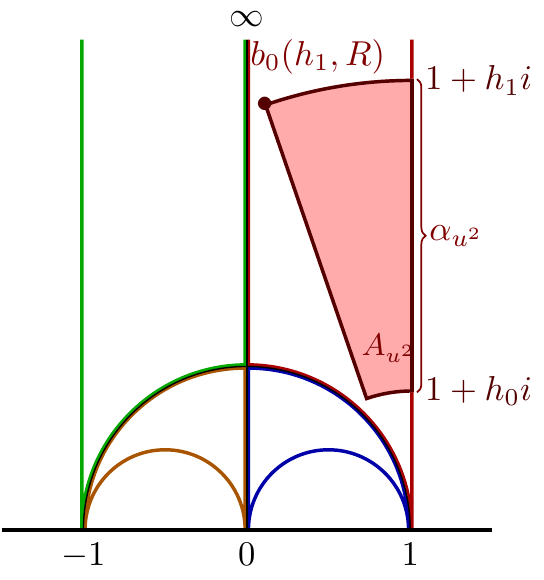}
\caption{Transition zone $A_{u^2}$ and $b_0(h_1,R)\in T_0$.}
\label{figu:symmetric transition zones}
\end{figure}

Then, the condition $b_0(h,R)\in T_0$ is equivalent to $2\tanh^2(R) \leq h \leq \coth(R)$. 
Thus, $L=\log(h_1/h_0)$ and $R$ are admissible if $2\tanh^2(R) \leq h_0 < h_1 \leq \coth(R)$.
That is, if $e^L \leq \coth(R)/2\tanh^2(R)$ or, equivalently, if \[2e^L \tanh^3(R) \leq 1.\pushQED{\qed}\qedhere\]

\setcounter{figure}{0}
\section{Estimates for the combinatorial spectrum}
\label{sect:estimates combinatorial spectrum}

In this appendix we estimate from below the bottom of the combinatorial spectrum $\mu_0(G,S)$, for the group $G<\pslr$ associated to the system of generators $S = \{u,{}^{t\!}u^3\}$, where $u=\begin{psmallmatrix} 1 & 1 \\ 0 & 1 \end{psmallmatrix}$ and ${}^{t\!}u$ is its transpose. By combinatorial spectrum, we refer to the spectrum of the combinatorial Laplace operator on the Cayley graph.

We estimate $\mu_0(G,S)$ from below following ideas of Nagnibeda~\cite{Na} and prove the following.

\begin{theo} \label{theo:combinatorial spectrum}
Let $G$ be the subgroup of $\pslr$ generated by $S = \{u, {}^{t\!}u^3\}$. Then, the bottom of the combinatorial spectrum of $G$ associated to $S$ satisfies \[\mu_0(G,S) > 0.4647.\]
\end{theo}

\begin{rema}
It can be proved that the bottom of the combinatorial spectrum associated to a symmetric finite system of $k>1$ generators, is bounded from above by $k-2\sqrt{k-1}$ (which corresponds to the bottom of the combinatorial spectrum of a regular tree of degree $k$).
In our case, this means that $\mu_0(G,S) < 4-2\sqrt{3}$ or, numerically, $\mu_0(G,S) < 0.5359$. In particular, this shows that the error in our estimate is less than $16\%$.
\end{rema}

In the following, we recall some aspects of combinatorial group theory we need and, in particular, we recall the definition of the bottom of the combinatorial spectrum $\mu_0(G,S)$. The following discussion is completely general.

\subsection{Combinatorial group theory}

Let $G$ be any group, and let $S$ be a subset of $G$. A \emph{word} in $S$ is any expression of the form
\[ w=s_{1}^{\sigma _{1}}s_{2}^{\sigma _{2}}\cdots s_{n}^{\sigma _{n}}\]
where $s_1,\dots,s_n\in S$ and $\sigma_i \in \{+1,-1\}$, $i=1,\dots,n$. The number $l(w)=n$ is the \emph{length} of the word.

Each word in $S$ represents an element of $G$, namely the product of the expression.
The identity element can be represented by the empty word, which is the unique word of length zero.

\subsubsection*{Notation} 
We use an overline to denote inverses, thus $\bar s$ stands for $s^{-1}$.

In these terms, a subset $S$ of a group $G$ is a system of generators if and only if every element of $G$ can be represented by a word in $S$. Henceforth, let $S$ be a fixed system of generators of $G$ and a word is assumed to be a word in $S$. A \emph{relator} is a non-empty word that represent the identity element of $G$.

Any word in which a generator appears next to its own inverse ($s\bar s$ or $\bar ss$) can be simplified by omitting the redundant pair. We say that a word is \emph{reduced} if it contains no such redundant pairs.

Let $v,w$ be two words. We say that $v$ is a \emph{subword} of $w$ if $w=v'vv''$, for some words $v',v''$. If $v'$ is the empty word we say that $v$ is a \emph{prefix} of $w$. If $v''$ is the empty word we say that $v$ is a \emph{suffix} of $w$.

We say that a word is \emph{reduced in $G$} if it has no non-empty relators as subword.
In particular, if a word is reduced in $G$, any of its subwords is also reduced in $G$.

For an element $g\in G$, we consider the \emph{word norm} $|g|$ to be the least length of a word which is equals to $g$ when considered as a product in $G$, and every such word is called a \emph{path}, that is, if its length coincides with its word norm when considered as a product in $G$. In particular, a path is always reduced in $G$. Moreover, a subword of a path is also a path. We say that two words are \emph{equivalent} if they represent the same element in $G$.

For a relator, we call a subword that is a relator, a \emph{subrelator}. We say that a relator is \emph{primitive} if every proper subword is reduced in $G$, that is, if it does not contain proper subrelators. In particular, a word is reduced in $G$ if and only if it contains no primitive relators as subword.
Note that, if $P$ is the set of all primitive relators, then $\langle S \mid P \rangle$ is a presentation of $G$.

The following elementary results (see Figure~\ref{figu:decomposition}) will be useful in \S\ref{sect:compatible type function}.

\begin{lemm}
Let $v,w$ be two different equivalent paths. Then, there are paths $v_0,v_1$, $w_0,w_1$ and $x$ such that $v = v_0v_1x$ and $w = w_0w_1x$, and $v_1\bar w_1$ is a primitive relator (of even length).
\end{lemm}

\begin{figure}[ht]
\centering\includegraphics[height=.275\textheight]{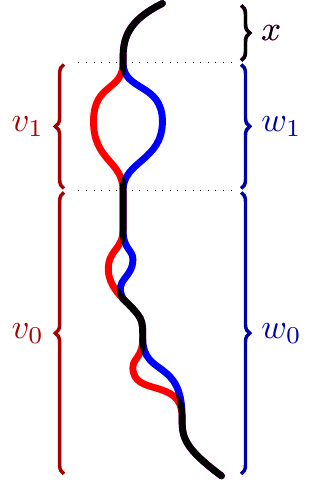}
\caption{Decomposition of two equivalent paths}
\label{figu:decomposition}
\end{figure}

\begin{proof}

Let $x$ be the largest common suffix of $v$ and $w$ (possibly $x$ is empty). Write $v=v'x$ and $w = w'x$.
Let $w_1$ and $v_1$ be the smallest non-empty suffixes of $w'$ and $v'$ respectively such that $v_1$ and $w_1$ are equivalent. Such $v_1$ and $w_1$ exist since $v$ and $w$ are different words. Moreover, they have the same length since they are equivalent, that is, they are paths that evaluate to the same element in $G$. Write $v'=v_0v_1$ and $w'=w_0w_1$ (possibly $v_0$ and $w_0$ are empty). In particular $v_0$ and $w_0$ are equivalent, since the same holds for $v',w'$ and $v_1,w_1$.

It remains to prove that $v_1\bar w_1$ is primitive.
Suppose $z$ is a subrelator of $v_1\bar w_1$.
Since $v_1$ and $w_1$ are paths, they are in particular reduced in $G$ and also their subwords. Then $z = v_2\bar w_2$ for some non-empty suffixes $v_2$ and $w_2$ of $v_1$ and $w_1$ respectively. In particular, $v_2$ and $w_2$ are non-empty suffixes of $w'$ and $v'$ respectively and $v_2,w_2$ are equivalent. But, by definition, $v_1$ and $w_1$ are the smallest such suffixes and therefore $v_2=v_1$ and $w_2=w_1$. Thus, $v_1\bar w_1$ has no proper subrelators and therefore, $v_1\bar w_1$ is primitive.
\end{proof}

As a direct consequence of the previous lemma, we have the following.

\begin{coro} \label{coro:primitive relator subword}
Let $v=v'yx$ and $w=w'zx$ be two equivalent paths such that $y\bar z$ is reduced in $G$. Then, $y\bar z$ is a subword of some primitive relator (of even length).
\end{coro}

\begin{proof}
Consider the decomposition given by the previous lemma. It is clear that $y$ is a subword of $v_1$ and $z$, of $w_1$. Then $y\bar z$ is a subword of the primitive relator $v_1\bar w_1$.
\end{proof}

\subsection{Combinatorial spectrum}
Let $G$ be a finitely generated group and $S\subset G$ be a finite system of generators of $G$.
Let $\ell^2(G)$ be the space of square-summable sequences on $G$ with the inner product \[\langle h,h' \rangle \coloneqq \sum_{g\in G} h_g {h\mathrlap{'}}_g,\] for $h,h'\in\ell^2(G)$, and define $\Delta_S:\ell^2(G)\to\ell^2(G)$, the combinatorial Laplace operator on $G$ associated to $S$, by \[(\Delta_S h)_g \coloneqq \sum_{s\in S\cup\bar S} (h_g - h_{gs}),\] for $h\in\ell^2(G)$.
Then, we define $\mu_0(G,S)$, \emph{the bottom of the combinatorial spectrum of $G$ associated to $S$} to be the bottom of the spectrum of $\Delta_S$, that is, \[\mu_0(G,S)\coloneqq \inf \left\{\frac{\langle \Delta_S h, h \rangle}{\langle h, h \rangle},\; h\in \ell^2(G)\right\}.\]

\begin{rema} The subjacent object in this discussion is the Laplace operator on the Cayley graph of $G$ associated to $S$. However we do not explain this here.
\end{rema}

\subsubsection{Nagnibeda's ideas}

In order to give estimates from below to the combinatorial spectrum we follow ideas of Nagnibeda~\cite{Na}, which are based in the following result, whose proof is elementary (see, for example,~\cite[\S7.1]{Co}).

\begin{prop}[Gabber--Galil's lemma] \label{prop:GG}
Let $G$ be a finitely generated group and $S$ a finite symmetric system of generators of $G$. Suppose there exists a function $L:G\times S\to \R_+$ such that, for every $g\in G$ and $s\in S$,
\[L(g,s) = \frac{1}{L(gs,s^{-1})} \qquad \text{ and } \qquad \sum_{s\in S} L(g,s) \leq k,\] for some $k>0$.
Then, \[\mu_0(G,S) \geq \# S - k.\]
\end{prop}

Let $S$ be a symmetric finite system of generators of $G$. For $g\in G$, denote by $|g|$ the word norm with respect to $S$ and define $S^\pm(g)\coloneqq \{s\in S :\; |gs| = |g| \pm 1\}$. For $g\in G$ and $s\in S$, we say that $gs$ is a \emph{successor} of $g$ if $s\in S^+(g)$ and that $gs$ is a \emph{predecessor} of $g$ if $s\in S^-(g)$. Henceforth we assume $S^{+}(g)\cup S^{-}(g)=S$, for every $g\in G$. Note that this is equivalent to say that every relator has even length.

A function $t:G \to \N$ is called a \emph{type function on $G$} and its value $t(g)$ at $g\in G$ is called the \emph{type of $g$}. We say that a type function $t$ is \emph{compatible with $S$}, or simply that $t$ is a \emph{compatible type function}, if the following two conditions are equivalent:
\begin{enumerate}
\item $t(g) = t(g')$;
\item $\#\{s\in S^+(g) :\; t(gs)=k\} = \#\{s'\in S^+(g') :\; t(g's')=k\}$, for every $k\in \N$.
\end{enumerate}
Equivalently, $t$ is a compatible type function if the (multiset of) types of successors of an element $g\in G$ (is/)are completely defined by its type $t(g)$.

For any type function $t:G \to \N$ and positive valuation $c: \N\to \R_+$, we can consider a function $L_c:G\times S\to \R_+$ defined by
\[L_c(g,s) = \begin{cases} c_k, & \text{if } s \in S^+(g),\;k = t(gs), \\ 1/c_k, & \text{if } s\in S^-(g),\;k = t(g). \end{cases}\]
It is clear then, by the definition, that any $L_c:G\times S\to \R_+$ defined as above satisfies $L_c(g,s) = 1/L_c(gs,s^{-1})$, since $s\in S^+(g)$ if and only if $s^{-1}\in S^-(gs)$, and $S=S^+(g)\cup S^-(g)$, for every $g\in G$.

Moreover, for a compatible type function $t$, we define for $k=t(g)\in \N$, $g\in G$,
\[
f_k(c) \coloneqq \sum_{s\in S} L_c(g,s) = \sum_{s\in S^+(g)} c_{t(gs)} + \frac{\# S^-(g)}{c_k}.
\]
Note that this is well defined since $t$ is compatible with $S$ and therefore the sum depends only on $k$, the type of $g$.

As a direct consequence of Gabber--Galil's lemma (Proposition~\ref{prop:GG}), we get the following.

\begin{coro} \label{coro:GG} Let $t:G\to \{0,\dots, K\}$ be a compatible type function. Then, \[\mu_0(G,S) \geq \# S - \max_{k=0,\dots,K} f_k(c),\] for every $c:\{0,\dots, K\} \to\R_+$, where $f_k$ is defined as above.
\end{coro}

Then, every compatible (finite) type function gives lower bounds for the combinatorial spectrum.

\subsection{Compatible type functions in our particular case}
\label{sect:compatible type function}

Until now, the discussion is completely general. We now specialize to the case of $G<\pslz$ with generators $u=\begin{psmallmatrix} 1 & 1 \\ 0 & 1 \end{psmallmatrix}$ and $v={}^{t\!}u^3$.
The aim in the following is to give a compatible finite type function in this case, in order to give estimates for the bottom of the combinatorial spectrum with the aid of Corollary~\ref{coro:GG}.
For this, we define a \emph{suffix type function} and prove that it is compatible with $S=\{u,v\}$.

It is not difficult to see that $\langle u,v \mid (u \bar v)^3\rangle$ is a presentation of $G$ and the set of primitive relators is given by
\[\{ (u \bar v)^3, (\bar u v)^3, (v \bar u)^3, (\bar v u)^3\}.\]
In particular, every relator has even length and we can apply previous discussion.

Let $\mbf S(g)$ be the set of all suffixes of paths for $g\in G$. Then, by the description of the primitive relators, as a direct consequence of Corollary~\ref{coro:primitive relator subword}, we have the following.

\begin{coro}
Let $s\in S$ and $r\in S\setminus\{s,\bar s\}$. The following cases cannot happen:
\begin{itemize}
\item $s,\bar s \in \mbf S(g)$;
\item $sr,\bar sr \in \mbf S(g)$;
\item $s, r^2 \in \mbf S(g)$;
\item $sr,s \in \mbf S(g)$;
\item $u, \bar v \in \mbf S(g)$;
\item $\bar u, v \in \mbf S(g)$;
\end{itemize}
\end{coro}

\begin{proof}
Neither $s^2$, $sr\bar s$, $uv$ nor $vu$ are subwords of a primitive relator.
\end{proof}

Let $\mbf S_n(g)$ be the set of all suffixes of length $n\in \N$ of paths for $g\in G$ and define, by recurrence, $\mbf S_1^*(g) = \mbf S_1(g)$ and
\[\mbf S_{n+1}^*(g) = \begin{cases} \mbf S_{n+1}(g) & \text{if } S_{n+1}(g) \neq \emptyset, \\ \mbf S_n^*(g) & \text{if } S_{n+1}(g) = \emptyset. \end{cases}\]
Note that any injective function $j:\mbf S_n^*(G) \to \N$ defines a (finite) type function $t=j\circ \mbf S_n^*:G\to\N$, which we call \emph{suffix type function of level~$n$}.

\begin{lemm} Let $t:G\to \N$ be a suffix type function of level~$2$. Then, it is compatible with $S$.
\end{lemm}

\begin{proof}
Being compatible with $S$ means that the type $t(g)$ of $g\in G$ completely defines the types of its successors. Then, it is enough to show that $\mbf S_2^*(g)$ defines completely the multiset $\{\mbf S_2^*(gs): s\in S^+(g)\}$.

From the previous corollary, we can deduce that $\mbf S_1^*(g) =\emptyset,\{u\},\{\bar u\},\{v\},\{\bar v\},\{u,v\}$ or $\{\bar u,\bar v\}$, and that
\[\mbf S_2^*(g) \in \{\emptyset\} \cup \{
  \{s\},\{s^2\},\{sr\}
  \}_{s\in S, r\in S\setminus\{s,\bar s\}} \cup \{
  \{\bar a b,\bar b a\},
  \{a^2, ba\}
  \}_{\{a,b\}\in\{\{u,v\},\{\bar u, \bar v\}\}}
  .\]
Moreover, it is clear that $s\in S^+(g)$ if and only if $\bar s\notin \mbf S_1(g)$.

Let $s\in S$, $r\in S\setminus\{s,\bar s\}$ and $\{a,b\}\in\{\{u,v\},\{\bar u, \bar v\}\}$.
\begin{itemize}
\item If $\mbf S_2^*(g) = \emptyset$, $g=id$ and evidently $\mbf S_2^*(ge) = \{e\}$, for $e\in S=S^+(g)$.
\item If $\mbf S_2^*(g) = \{s\}$ or $\{s^2\}$, then $\mbf S_2^*(ge) = \{se\}$ for $e\in S\setminus\{\bar s\}=S^+(g)$.
\item If $\mbf S_2^*(g) = \{ab\}$, then $\mbf S_2^*(ge) = \{be\}$, for $e\in S\setminus\{b\}=S^+(g)$.
\item If $\mbf S_2^*(g) = \{a\bar b\}$, then $S^+(g)=S\setminus\{\bar b\}$, $\mbf S_2^*(ga) = \{\bar b a,\bar a b\}$ and $\mbf S_2^*(gq) = \{bq\}$, for $q\in \{\bar a,b\}=S^+(g)\setminus\{a\}$.
\item If $\mbf S_2^*(g) = \{\bar a b,\bar b a\}$, then $\mbf S_2^*(ge) = \{ae,be\}$, for $e\in \{a,b\}=S^+(g)$.
\item If $\mbf S_2^*(g) = \{a^2, ba\}$, then $\mbf S_2^*(ge) = \{ae\}$, for $e\in S\setminus\{\bar a\}=S^+(g)$.
\end{itemize}
Thus, given only the value of $\mbf S_2^*(g)$ we can tell the corresponding value of $\mbf S_2^*(gs)$ for each $s\in S^+(g)$ and therefore, suffix type functions are compatible with $S$.
\end{proof}

We summarize the proof of the previous lemma by the following diagram which shows each possible $\mbf S_2^*(g)$, $g\in G$ with its respective multiset of $\mbf S_2^*(ge)$, $e\in S^+(g)$:
\begin{align*}
\mbf S_2^*(g) & \to \mbf S_2^*(ge), e\in S^+(g) \\
\\
\emptyset & \to \{u\}, \{\bar u\}, \{v\}, \{\bar v\} \\
\{s\} & \to \{s^2\}, \{sr\}, \{s \bar r\} \\
\{s^2\} & \to \{s^2\}, \{sr\}, \{s \bar r\} \\
\{ab\} & \to \{b^2\}, \{ba\}, \{b \bar a\} \\
\{a \bar b\} & \to \{\bar b^2\}, \{\bar b a, \bar a b\}, \{\bar b \bar a\} \\
\{\bar b a, \bar a b\} & \to \{a^2,ba\}, \{ab,b^2\} \\
\{a^2,ba\} & \to \{a^2\}, \{ab\}, \{a \bar b\},
\end{align*}
where $s\in S$, $r\in S\setminus\{s,\bar s\}$ and $\{a,b\}\in\{\{u,v\},\{\bar u, \bar v\}\}$.

It is not difficult to see in the previous diagram that there are different suffix types which share the types of the successors. Namely $\{a\}$, $\{a^2\}$, $\{ba\}$ and $\{ba,a^2\}$. This allows us to reduce the number of types. Furthermore, it is clear that distinguishing $a$ and $\bar a$ or $a$ and $b$ in the previous description has no major benefit.
This motivates the definition of the following type function. Let $T:G\to\{0,\dots,3\}$ be the type function defined as follows:
\[
T(g) =
\begin{cases}
0 & \text{if } \mbf S_2^*(g) = \emptyset, \\
1 & \text{if } \mbf S_2^*(g) = \{a\},\; \{a^2\},\; \{ba\} \text{ or } \{ba,a^2\}, \\
2 & \text{if } \mbf S_2^*(g) = \{\bar b a\}, \\
3 & \text{if } \mbf S_2^*(g) = \{\bar a b, \bar b a\},
\end{cases}
\]
for $\{a,b\}\in\{\{u,v\},\{\bar u, \bar v\}\}$.

From the previous discussion, we deduce the following.
\begin{theo} \label{theo:cone types} The type function $T:G\to \{0,\dots,3\}$ is compatible with $S$.
Moreover,
\begin{itemize}
\item Type $0$ elements have four type $1$ successors;
\item Type $1$ elements have two type $1$ and one type $2$ successors;
\item Type $2$ elements have two type $1$ and one type $3$ successor; and
\item Type $3$ elements have two type $1$ successor.
\end{itemize}
\end{theo}

Thus, we have a compatible type function with a full description of the types of the successors for each type. We can then finally apply Nagnibeda's ideas (Corollary~\ref{coro:GG}) to give estimates for the bottom of the combinatorial spectrum.

\subsection{Estimates for the bottom of the combinatorial spectrum} By Theorem~\ref{theo:cone types}, the $f_k$ of Corollary~\ref{coro:GG} are given by:
\begin{itemize}
\item $f_0(c) = 4c_1$;
\item $f_1(c) = 2 c_1 + c_2 + 1/c_1$;
\item $f_2(c) = 2 c_1 + c_3 + 1/c_2$; and
\item $f_3(c) = 2 c_1 + 1/c_3$.
\end{itemize}
It follows that $\mu_0(G,S) \geq \# S - \max_{k} f_k(c)$, for every $c=(c_1,c_2,c_3)\in \R_+^{3}$.
Thus, the problem can be reduced to find the optimal such bound.
This can be solved numerically: we get that $\bar c\in\R_+^3$ with
\[\bar c_1= 0.5680;\; \bar c_2 \approx 0.6387;\; \bar c_3 \approx 0.8336,\]
is a (local) minimun for $\max_{k} f_k(c)$, and $\max_{k} f_k(\bar c) \approx 3.5353$.

Finally, since $\# S=4$, it follows that \[\mu_0(G,S) > 0.4647.\]
This concludes the proof of Theorem~\ref{theo:combinatorial spectrum}\qed



\end{document}